\documentclass[11pt]{article}

\usepackage[left=1.2in,right=1.2in]{geometry}

\usepackage[utf8]{inputenc}
\usepackage[T1]{fontenc}
\usepackage{ifthen}
\usepackage{amssymb}
\usepackage{amsthm} 
\usepackage{bbm}
\usepackage{enumitem}

\usepackage{amsmath}
\usepackage{blkarray}

\usepackage[all]{xy}

\usepackage{xcolor}
\definecolor{notecolor}{rgb}{1,0,0}

\usepackage{tikz}
\usepackage{tkz-graph}
\usepackage{tikz-cd}
\usetikzlibrary{arrows.meta, positioning}
\usepackage{xparse}
\usepackage{mathrsfs}
\usepackage{graphicx}
\usepackage{hyperref}

\definecolor{linkcolor}{rgb}{0,0,0.8} 

{
\theoremstyle{plain}
\newtheorem{theorem}{Theorem}[section]
\newtheorem{lemma}[theorem]{Lemma}
\newtheorem{proposition}[theorem]{Proposition}
\newtheorem{corollary}[theorem]{Corollary}
\newtheorem{example}[theorem]{Example}
\newtheorem{definition}[theorem]{Definition}
\newtheorem{conjecture}[theorem]{Conjecture}
\newtheorem*{conjecture*}{Conjecture}

}

{
\theoremstyle{definition}

\newtheorem{remark}[theorem]{Remark}
}

\newcommand{\sizedescriptor}[2]
{
\ifthenelse{\equal{#1}{0}}{}{
\ifthenelse{\equal{#1}{1}}{\big}{
\ifthenelse{\equal{#1}{2}}{\Big}{
\ifthenelse{\equal{#1}{3}}{\bigg}{
\ifthenelse{\equal{#1}{4}}{\Bigg}{
#2}}}}}
}

\newcommand{\proven}[1]{\underline{#1} \\ \ \vspace{-2ex} \\}

\newcommand{\impl}{\Rightarrow}  
\newcommand{\all}[1]{\forall #1 .\,}  

\newcommand{\df}[1]{\emph{\textbf{#1}}}  

\NewDocumentCommand{\set}
	{O{auto} m G{\empty}}
	{\sizedescriptor{#1}{\left}\{ {#2} \ifthenelse{\equal{#3}{}}{}{ \; \sizedescriptor{#1}{\middle}| \; {#3}} \sizedescriptor{#1}{\right}\}}

\newcommand{\pst}{\mathcal{P}}  

\newcommand{\NN}{\mathbb{N}}

\newcommand{\RR}{\mathbb{R}}

\newcommand{\intoo}[3][\RR]{{#1}_{(#2, #3)}}
\newcommand{\intcc}[3][\RR]{{#1}_{[#2, #3]}}

\newcommand{\intco}[3][\RR]{{#1}_{[#2, #3)}}

\newcommand{\er}{\overline{\RR}}

\NewDocumentCommand{\oball}  
	{O{\empty} G{\empty} G{\empty}}
	{B_{#1}\ifthenelse{\equal{#2}{}}{}{\!\left(#2, #3\right)}}
\NewDocumentCommand{\cball}  
	{O{\empty} G{\empty} G{\empty}}
	{\overline{B}_{#1}\ifthenelse{\equal{#2}{}}{}{\!\left(#2, #3\right)}}
\NewDocumentCommand{\cth}  
	{O{\empty} G{\empty} G{\empty}}
	{\overline{\mathrm{th}}_{#1}\ifthenelse{\equal{#2}{}}{}{\!\left(#2, #3\right)}}

\newcommand{\rstr}[1]{\left.{#1}\right|}  
\newcommand{\parto}{\mathrel{\rightharpoonup}}  
\newcommand{\prj}{\mathrm{pr}}
\NewDocumentCommand{\dimg}  
	{O{\empty} m G{\empty}}
	{{#2}_*\ifthenelse{\equal{#3}{}}{}{\!\sizedescriptor{#1}{\left}( {#3} \sizedescriptor{#1}{\right})}}
\NewDocumentCommand{\pimg}  
	{O{\empty} m G{\empty}}
	{{#2}^*\ifthenelse{\equal{#3}{}}{}{\!\sizedescriptor{#1}{\left}( {#3} \sizedescriptor{#1}{\right})}}

\usepackage{mathtools}
\mathtoolsset{centercolon}

\newcommand{\mg}[1][]{\mathsf{mag}_{#1}}

\newcommand{\fsub}{\mathrm{Fin}}
\newcommand{\ifsub}{\fsub_{+}}
\newcommand{\csub}{\mathrm{Cmp}}
\newcommand{\icsub}{\csub_{+}}

\newcommand{\card}[1]{\mathnormal{\#}{#1}}

\newcommand{\sgn}{\mathrm{sgn}}
\newcommand{\decr}{\mathop{\text{\tiny$\searrow$}}}

\newcommand{\cc}{\overline{C}}
\newcommand{\skw}{\mathrm{skew}}
\newcommand{\cn}[2][F]{{\urcorner}{#1}_{#2}}

\title{Continuity of Magnitude at Skew Finite Subsets of $\ell_1^N$}

\author{
Sara Kali\v{s}nik\thanks{
Pennsylvania State University,
\texttt{skalisnik@psu.edu}}\phantom{x} and 
Davorin Le\v{s}nik\thanks{
University of Ljubljana,
\texttt{davorin.lesnik@fmf.uni-lj.si}}
}

\date{}
\begin{document}

\maketitle

\begin{abstract}
Magnitude is an isometric invariant of metric spaces introduced by Leinster. Although magnitude is nowhere continuous on the Gromov–Hausdorff space of finite metric spaces, continuity results are possible if we restrict the ambient space.
In this paper, we focus on~$\ell_1^N$ and prove that magnitude is continuous at every skew finite subset of~$\ell_1^N$, that is, at every finite set whose coordinate projections are injective.
For such sets, we analyze cubical thickenings and derive an explicit formula for their weight measures. This yields a formula for the magnitude of these thickenings, which we use to prove that their magnitude converges to that of the underlying finite set.
Since skew finite subsets of~$\ell_1^N$ form an open and dense subset of the space of all finite subsets, magnitude is continuous on an open dense subset of the space of finite subsets of $\ell_1^N$.
\end{abstract}

\tableofcontents

\section{Introduction}

Magnitude is a cardinality-like invariant of metric spaces introduced by Leinster in the 2010s~\cite{leinster2010}. Originally arising in category theory as an extension of Euler characteristic to enriched categories, it was soon recognized as a geometric invariant of compact metric spaces~\cite{leinster2010, LW13, Willerton2014, W09, meckes2013}. 
For example, it has been shown to encode volume, capacity, dimension, and intrinsic volumes~\cite{leinster2010, Meckes2015, Willerton2014, LW13}.
It has also found applications in diverse areas: in ecology as a quantitative measure of biodiversity~\cite{Solow1994, LC12}, in machine learning as a descriptor of intrinsic diversity in latent representations~\cite{magnitude-ML}, and in topological data analysis through connections with persistent homology~\cite{O18, OMALLEY2023107396, GH21}. 

From both theoretical and applied perspectives, understanding the continuity properties of magnitude is of central importance. 
When one considers the space of isometry classes of finite metric spaces equipped with the Gromov--Hausdorff metric, magnitude fails to be continuous~\cite{leinster2010, Roff25}. In fact, it is shown in~\cite{katsumasa2025magnitudegenericallycontinuousfinite} that magnitude is nowhere continuous on this space. Nevertheless, there are indications that when restricting to suitable subclasses of metric spaces, magnitude is continuous.

In this paper, the subclass we focus on are subspaces of~$\ell_1^N$. There are two reasons for this. The first is that this is an interesting and important class of spaces in its own right, especially in the context of magnitude which interacts particularly nicely with the $1$-metric (for example, the magnitude of the $1$-product of two metric spaces is the product of their magnitudes~\cite[Proposition 2.3.6]{leinster2010}). The second is that a proof of continuity of magnitude in spaces $\ell_1^N$ may allow for generalization to a far greater class of spaces. Recall that in~\cite[Theorem~3.1]{leinster2023spaces}, Leinster and Meckes prove the one-point property (essentially a special case of continuity) for all metric spaces that can be isometrically embedded into a finite-dimensional vector subspace of~$L_1$. Their result covers many of the one-point property results that appear in the literature for specific classes of metric spaces (such as Euclidean spaces). The proof proceeds by reducing the problem to the one-point property of subspaces of~$\ell_1^N$. It is our hope that something similar could be done for continuity in general, i.e.\ that proving (a sufficiently strong form of) continuity of magnitude in all~$\ell_1^N$ would imply continuity in all finite-dimensional subspaces of~$L_1$.

Our first contribution is a general criterion that reduces continuity of magnitude at a compact set to a one-sided estimate along metric thickenings. More precisely, in any tractable metric space~$M$ (in the sense of~\cite{kalisnik2025tractablemetricspacescontinuity}), continuity of magnitude at compact $K \subseteq M$ is equivalent to the convergence
\[
\lim_{r \decr 0} \mg\big(\cth[M]{K}{r}\big) = \mg(K)
\]
and, equivalently, to the existence of arbitrarily small $r > 0$ with $\mg\big(\cth[M]{K}{r}\big)$ arbitrarily close to $\mg(K)$ from above (Lemma~\ref{lemma:magnitude-point-continuity-characterization-in-tractable-spaces}). 

We then specialize to $M = \ell_1^N$, where it is technically convenient to work with cubical thickenings in the $\infty$-metric: for $p \in \ell_1^N$ and $r > 0$ we write $\cc_p(r) := \cball[\infty]{p}{r}$ and, for a finite set $F \subseteq \ell_1^N$, we set $\cc_F(r) := \bigcup_{p \in F} \cc_p(r)$. Our main technical result is an explicit description of the weight measure of~$\cc_F(r)$ when $F$ is \df{skew}, i.e.\ when the coordinate projections $\prj_k|_F$ are injective. In this case, for all sufficiently small~$r$ the cubes in $\cc_F(r)$ have pairwise disjoint coordinate projections, and we prove that $\cc_F(r)$ admits a weight measure of the form
\[
\omega_{\cc_F(r)}
\ = \
\Big(\tfrac{1}{2^N} \sum_{D = 0}^N \lambda_{\cc_F(r)^{(D)}}^D\Big)
\ -\!\!
\sum_{\substack{p \in F \\ s \in \set{-1, 1}^N}}\!\!\!\!\alpha_{p, s}(r)\,\delta_{p + r s},
\]
where the coefficients $\alpha_{p, s}(r)$ are determined by a linear system of equations indexed by the vertices of~$\cc_F(r)$ (Theorem~\ref{theorem:cubes-weight-measure}), $\cc_F(r)^{(D)}$ denotes the union of the $D$-dimensional faces of the cubes in $\cc_F(r)$, $\lambda_{\cc_F(r)^{(D)}}^D$ is the $D$-dimensional Lebesgue measure restricted to these faces, and $\delta_{p + r s}$ denotes the Dirac measure at the vertex $p + r s$.

This yields an explicit formula for $\mg\big(\cc_F(r)\big)$. Using this cube weight measure formula, we deduce that
\[
\lim_{r \decr 0} \mg\big(\cc_F(r)\big) = \mg(F)
\qquad
\text{for every skew finite $F \subseteq \ell_1^N$}.
\]
Hence, magnitude in~$\ell_1^N$ is continuous at every skew finite subset (Theorem~\ref{theorem:skew-magnitude-continuity}). Since skew finite subsets of~$\ell_1^N$ form an open and dense subset in the set of all finite subsets of~$\ell_1^N$, the finite space version of magnitude in~$\ell_1^N$ is continuous ``almost everywhere''.

\subsection{Notation}\label{sub:notation}

We denote the set of natural numbers by~$\NN$. We treat~$0$ as a natural number, so $\NN = \set{0, 1, 2, 3,\ldots}$.

We denote the set of real numbers by~$\RR$, and the set of extended real numbers (with infinities included) by~$\er$. That is, $\er = \RR \cup \set{-\infty, \infty}$.

Subsets, given by a relation, are denoted by that relation in the index; for example, $\RR_{\geq 0}$ is the set of non-negative real numbers. Intervals between two numbers are denoted by these two numbers in brackets and in the index. Round, or open, brackets $(\ )$ denote the absence of the boundary in the set, and square, or closed, brackets $[\ ]$ its presence. For example, $\intcc{0}{1} = \set{x \in \RR}{0 \leq x \leq 1}$ is the usual closed unit interval, and $\intco[\NN]{5}{10} = \set{n \in \NN}{5 \leq n < 10} = \set{5, 6, 7, 8, 9}$.

For any set~$A$, we denote its cardinality by~$\card{A}$.

The powerset of a set~$A$ (the set of all subsets of~$A$) is denoted by~$\pst(A)$. We denote the set of finite subsets of~$A$ by~$\fsub(A)$, and the set of non-empty finite subsets by~$\ifsub(A)$. If $A$ is also equipped with a topology, then $\csub(A)$ denotes the set of compact subspaces of~$A$, and $\icsub(A)$ the set of non-empty compact subspaces of~$A$.

A function $f$ mapping from a set~$A$ to a set~$B$ is denoted as $f\colon A \to B$. If $f$ is merely a partial map (not necessarily defined on the whole~$A$), we write this as $f\colon A \parto B$.


Following the notation from~\cite{leinster2023spaces}, for any $p \in \er_{\geq 1}$ and $N \in \NN$, let $\ell_p^N$ denote the Banach space~$\RR^N$, equipped with the $p$-norm (and the induced $p$-metric), and let $\ell_p$ denote the infinite-dimensional version of this space. Moreover, $L_1$ is the shorthand for $L_1\big(\intcc{0}{1}, \RR\big)$, i.e.\ the space of (equivalence classes of) measurable functions $\intcc{0}{1} \to \RR$, equipped with the usual integral $1$-norm.

Given a metric space~$M$ with a metric~$d$,
\begin{itemize}
\item
the open ball in~$M$ with the center in $x \in M$ and radius $r$ is denoted by $\oball[M]{x}{r}$, likewise for the closed ball $\cball[M]{x}{r}$ (and we shorten the notation for a ball in $p$-metric to $\oball[p]{x}{r}$, resp.~$\cball[p]{x}{r}$);
\item
the (closed) $r$-thickening of~$A$ in~$M$ for a subset $A \subseteq M$ and $r \in \RR_{\geq 0}$ is denoted as
\[\cth[M](A, r) := \set{x \in X}{d(x, A) \leq r}.\]
\end{itemize}

For any $p \in \RR^N$ and $r \in \RR_{> 0}$, we denote the (closed) cube with the center in~$p$ and ``radius''~$r$ by
\[\cc_p(r) \ := \ \cball[\infty]{p}{r} \ =\!\!\prod_{k \in \intcc[\NN]{1}{N}}\!\!\intcc{p_k - r}{p_k + r}.\]

We write $\displaystyle{\lim_{x \decr a} f(x)}$ for the right-sided limit (i.e.\ the limit when $x$ approaches~$a$ from above).

We use~$1_m$ to denote the vector that has $m$~components, all of which are equal to~$1$.

For any matrix~$A$, we write $\mathrm{sum}(A)$ for the sum of all its entries. This includes matrices with a single column, i.e.\ $\mathrm{sum}(v)$ is the sum of all components of a vector~$v$.

\section{Preliminaries}\label{section:prelims}

In this section we review some basic definitions and standard results concerning the magnitude of finite and compact (subspaces of) metric spaces~\cite{leinster2010, meckes2013}.

\subsection{The Magnitude of Finite Metric Spaces}\label{subsection:finite-metric-space-magnitude}

In this subsection, we recall the definition of magnitude of a finite metric space~\cite{leinster2010}.

\begin{definition}\label{definition:finite-metric-space-magnitude}
Let $F$ be a finite set, equipped with a metric~$d$, and let $m := \card{F}$. The matrix $Z_F \in \RR^{m \times m}$, given by $Z_F := \big(e^{-d(p, q)}\big)_{p, q \in F}$, is called the \df{similarity matrix} of~$F$. A vector $w \in \RR^m$ is a \df{weighting} for $Z_F$ when $Z_F \cdot w = 1_m$, where $1_m \in \RR^m$ is the vector with all components equal to~$1$. If a weighting~$w$ for $Z_F$ exists, the \df{magnitude} of the metric space~$F$ is defined to be
\[\mg(F) := \mathrm{sum}(w),\]
i.e.\ the sum of all components of the weighting~$w$ (this sum is independent of the choice of the weighting~\cite{leinster2010}).
\end{definition}

\begin{remark}
If $Z_F$ is invertible, there exists a unique weighting $w = Z_F^{-1} \cdot 1_{n}$, so the magnitude of~$F$ is defined and $\mg(F) = \mathrm{sum}(Z_F^{-1} \cdot 1_{n}) = \mathrm{sum}(Z_F^{-1})$ (the latter denotes the sum of all entries of the matrix~$Z_F^{-1}$).
\end{remark}

The notation that we use, i.e.~$\mg(F)$ as the magnitude of a finite metric space $F$, is a deviation from the more common notation~$|F|$, but $|F|$ clashes with the notation for the absolute value --- something which we very much use in this paper, given that we primarily consider spaces, equipped with the $1$-metric.

For any metric space~$M$, we can study the magnitude of all its finite subspaces. This gives us a partial function $\mg\colon \fsub(M) \parto \RR$ (where $\fsub(M)$ denotes the set of all finite subsets of~$M$). 
We consider $\fsub(M)$ to be equipped with the Hausdorff metric~$d_H$. This is only an extended metric, however, since the empty subset is at infinite distance from the others. For convenience, we restrict ourselves to~$\ifsub(M)$, i.e.\ the set of all \emph{non-empty} subsets of~$M$, where the Hausdorff metric is indeed a metric. This does not change any discussion on continuity: since $\emptyset$ is an isolated point in~$\fsub(M)$, any (partial) map with domain $\fsub(M)$ is continuous if and only if its restriction to~$\ifsub(M)$ is.

 To avoid lengthy notation when considering magnitude on different domains, we will denote that by putting the domain in the index, like this: $\mg[\fsub(M)]$, $\mg[\ifsub(M)]$.

\subsection{Positive Definite Metric Spaces}

Not every similarity matrix has a weighting, so the magnitude is not defined for every finite metric space.
One way to guarantee that a finite metric space has magnitude is to require that it has a positive definite similarity matrix; such a metric space is also called positive definite. It follows from Sylvester's criterion that its every subspace has this property as well. Hence, the following is a generalization of this property to not necessarily finite metric spaces.

\begin{definition}
A metric space~$M$ is \df{positive definite} when its every finite subspace has a positive definite similarity matrix.
\end{definition}

Magnitude has many additional desirable properties if we restrict ourselves to positive definite metric spaces~\cite{meckes2013}. If $M$ is positive definite, the similarity matrix $Z_F$ is invertible for every finite subspace of $M$, so the maps $\mg[\fsub(M)]$ and $\mg[\ifsub(M)]$ are total. Additionally, magnitude is \df{inclusion-monotone}, i.e.\ for all $F', F'' \in \fsub(M)$, if $F' \subseteq F''$, then $\mg(F') \leq \mg(F'')$. For the proof of this fact, see~\cite[
Corollary 2.4.4]{leinster2010}.

\subsection{The Magnitude of Compact Metric Spaces}

Because of inclusion-monotonicity of magnitude, there is a natural extension of the definition of magnitude to compact positive definite metric spaces. The idea is that a metric compact is totally bounded, and can thus be arbitrarily well approximated with finite subsets from below.

\begin{definition}\label{definition:compact-metric-space-magnitude}
Let $K$ be a compact positive definite metric space. Then its \df{magnitude} is defined as
\[\mg(K) := \sup\set[1]{\mg[\fsub(K) ](F)}{F \in \fsub(K)}.\]
\end{definition}

We thus get a partial\footnote{Even though $\mg[\fsub(M)]$ is total for positive definite~$M$, the supremum might be infinite, so $\mg(K)$ might not be a real number~\cite[Theorem~2.1]{leinster2023spaces}.} map $\mg\colon \csub(M) \parto \RR$ (or $\mg[\csub(M)]$ for short) where $\csub(M)$ stands for the set of all compact subspaces of~$M$. This set can again be equipped with the Hausdorff metric, or rather, it becomes a metric when we restrict ourselves to $\icsub(M)$ (the set of all \emph{non-empty} compact subspaces of~$M$). Again, because $\emptyset$ is an isolated point, there is no difference between continuity of~$\mg[\csub(M)]$ and~$\mg[\icsub(M)]$.

In principle, there is nothing stopping us to take this as the definition of the magnitude for any metric compact, but for general (non-positive definite) spaces, this is poorly behaved. In particular, for a finite metric space which is not positive definite, Definitions~\ref{definition:finite-metric-space-magnitude} and~\ref{definition:compact-metric-space-magnitude} might not yield the same magnitude~\cite{meckes2013}.

\begin{example}\label{example:interval-magnitude}
Let $a, b \in \RR$, $a \leq b$. The magnitude of the interval~$\intcc{a}{b}$~\cite[Theorem~7]{LW13} with respect to the standard Euclidean metric is given by
\[
\mg(\intcc{a}{b}) = 1 + \frac{b - a}{2}.
\]
More generally~\cite[Corollary~5.4.3]{LM17}, if \( A = \bigcup_{i \in \intcc[\NN]{1}{n}} \intcc{a_i}{b_i} \subseteq \mathbb{R} \) is a finite union of closed bounded intervals such that $a_i \leq b_i < a_{i+1} \leq b_{i+1}$ for all $i \in \intco[\NN]{1}{n}$, then the magnitude is given by
\[
\mg(A) = 1 + \sum_{i \in \intcc[\NN]{1}{n}} \frac{b_i - a_i}{2} + \sum_{i \in \intco[\NN]{1}{n}} \tanh\big(\frac{a_{i+1} - b_i}{2}\big).
\]
\end{example}

Other definitions of magnitude for compact metric spaces have appeared in the literature~\cite{leinster2010, LW13, Willerton2014, W09, meckes2013}. They have been shown to be equivalent for positive definite spaces, but not general ones~\cite{meckes2013}. One of these approaches, relevant for our manuscript, is via weight measures~\cite[Section 2.1]{Willerton2014}.

\begin{definition}
Let $(K, d)$ be a compact positive definite metric space. 
A finite signed Borel measure $\omega$ on $K$ is a \df{weight measure} 
when it satisfies
\[
\int_K e^{-d(x, a)} \, d\omega(x) = 1
\quad \text{for all } a \in K.
\]
If a compact metric space $K$ admits a weight measure $\omega$, 
then its magnitude is given by
\[
\mg(K) = \omega(K) = \int_K d\omega(x).
\]
\end{definition}

In the Euclidean setting, weight measures can frequently be written as linear combinations of Lebesgue and Dirac measures. Since one of the key statements of this manuscript is an explicit formula for the weight measures of unions of cubes in $\RR^N$, we now introduce notation for skeleta of polytopes and the corresponding Lebesgue and Dirac measures.

Let $P \subseteq \RR^N$ be a polytope (i.e.\ the convex hull of a finite set) or, more generally, a finite disjoint union of polytopes. For any $D \in \NN_{\leq N}$, let $P^{(D)}$ denote the $D$-skeleton of~$P$ (i.e.\ the union of all $D$-dimensional faces of~$P$).

Let $\lambda_S^D$ denote the $D$-dimensional Lebesgue measure on a measurable set $S \subseteq \RR^N$. In particular, we may consider $\lambda_{P^{(D)}}^D$, i.e.\ the $D$-dimensional Lebesgue measure on the $D$-skeleton of a (finite disjoint union of) polytope(s)~$P$.

For any point $p \in \RR^N$, let $\delta_p$ denote the Dirac measure at~$p$. Note that we have $\delta_p = \lambda_{\set{p}}^0$. More generally, if $F \in \fsub(\RR^N)$, then $\lambda_F^0 = \sum_{p \in F} \delta_p$.

\begin{example}\cite[
Theorem 2]{Willerton2014}\label{ex:weight-interval}
Let $a, b \in \RR$, $a \leq b$. The weight measure of the interval $\intcc{a}{b}$ is equal to 
\[
\omega_{\intcc{a}{b}} = \tfrac{1}{2} \big(\delta_a + \lambda_{\intcc{a}{b}}^1 + \delta_b\big).
\]
Note that the formula for the interval weight measure can be rewritten as $\omega_{\intcc{a}{b}} = \tfrac{1}{2} \big(\lambda_{\intcc{a}{b}^{(0)}}^0 + \lambda_{\intcc{a}{b}^{(1)}}^1\big)$.
\end{example}

\section{Magnitude Continuity via Thickenings}

In this section we recall some definitions and prove a general result, which we will later use for the main results of this paper. 

It is known that magnitude is in general not continuous, and it is a subject of active research to determine when it is. In~\cite{kalisnik2025tractablemetricspacescontinuity} we introduced a class of metric space, called tractable metric spaces, which we consider convenient for study of continuity of magnitude. We recall the definition here (for more details, how this is connected with magnitude continuity, see the aforementioned paper).

\begin{definition}
A metric space~$M$ is \df{tractable} when it is positive definite and every closed ball in it is compact and has finite magnitude.
\end{definition}

If $M$ is tractable, then $\mg[\csub(M)]$ and~$\mg[\icsub(M)]$ are total maps.

In any metric space~$M$ with a metric~$d$, the (closed) $r$-thickening of a subset $S \subseteq M$ is defined as
\[\cth[M]{S}{r} := \set[1]{x \in M}{d(S, x) \leq r}.\]
If $M$ is tractable, then for every compact $K \subseteq M$ and every $r \in \RR_{\geq 0}$ the thickening $\cth[M]{K}{r}$ is also compact and thus has (real-valued) magnitude~\cite[Lemma 3.1]{kalisnik2025tractablemetricspacescontinuity}.

We can use thickenings in tractable spaces to give the following characterization of continuity of magnitude at a point.

\begin{lemma}\label{lemma:magnitude-point-continuity-characterization-in-tractable-spaces}
Let $M$ be a tractable metric space and $K \in \icsub(M)$. The following statements are equivalent.
\begin{enumerate}[label=(\roman*)]
\item
The map $\mg[\icsub(M)]$ is continuous at~$K$.
\item
$\displaystyle{\lim_{r \decr 0} \mg\big(\cth[M]{K}{r}\big) = \mg(K)}$
\item
For every $\epsilon \in \RR_{> 0}$ there exists $\delta \in \RR_{> 0}$ such that
\[\mg\big(\cth[M]{K}{\delta}\big) < \mg(K) + \epsilon.\]
\end{enumerate}
\end{lemma}

\begin{proof}
\
\begin{itemize}
\item\proven{$(i \impl ii)$}
Observe that $0 \leq d_H\big(\cth[M]{K}{r}, K\big) \leq r$ for all $r \in \RR_{\geq 0}$, so $\lim_{r \decr 0} \cth[M]{K}{r} = K$ in~$\icsub(M)$. Since continuous maps preserve limits, the claim follows.
\item\proven{$(ii \impl iii)$}
Immediate from the definition of a limit.
\item\proven{$(iii \impl i)$}
Take any $\epsilon \in \RR_{> 0}$. By assumption there exists $\delta' \in \RR_{> 0}$ such that $\mg\big(\cth[M]{K}{\delta'}\big) < \mg(K) + \epsilon$. Let $L \in \icsub(M)$ be such that $d_H(K, L) < \delta'$. Since $L \subseteq \cth[M]{K}{d_H(K, L)} \subseteq \cth[M]{K}{\delta'}$, by inclusion-monotonicity of magnitude we get $\mg(L) < \mg(K) + \epsilon$. On the other hand, since magnitude is lower semicontinuous in positive definite (in particular tractable) metric spaces~\cite[Theorem 2.6]{meckes2013}, there exists $\delta'' \in \RR_{> 0}$ such that if $d_H(K, L) < \delta''$, then $\mg(L) > \mg(K) - \epsilon$. In conclusion, whenever $d(K, L) < \delta := \min\set{\delta', \delta''}$, we have $\big|\mg(K) - \mg(L)\big| < \epsilon$ which proves continuity of $\mg[\icsub(M)]$ at~$K$.
\end{itemize}
\end{proof}

\section{Cubes and Skew Finite Subsets of $\ell_1^N$}

Let us apply Lemma~\ref{lemma:magnitude-point-continuity-characterization-in-tractable-spaces} now for the special case of finite subsets of spaces~$\ell_1^N$ (which are tractable~\cite[Proposition 3.3]{kalisnik2025tractablemetricspacescontinuity}). The $r$-thickening of a finite $F \in \fsub(M)$ is just the finite union of closed balls with centers in the points of~$F$ and radius~$r$; in particular, for $F \in \fsub(\ell_1^N)$, it is the finite union of $1$-balls (diamonds). However, it is much easier to work with $\infty$-balls (cubes) in~$\ell_1^N$, hence we will phrase a sufficient condition for continuity in terms of cubes. But first, some notation.

\begin{definition}
For any $p = (p_k)_{k \in \intcc[\NN]{1}{N}} \in \ell_1^N$ and $r \in \RR_{\geq 0}$ we define the \df{(closed) cube}
\[\cc_p(r) \ := \ \cball[\infty]{p}{r} \ =\!\!\prod_{k \in \intcc[\NN]{1}{N}}\!\!\intcc{p_k - r}{p_k + r}.\]
\end{definition}
Here $\cball[\infty]$ stands for the closed ball in the $\infty$-metric, although we emphasize that we use this only as a convenient notation: this is still meant to be a subspace of~$\ell_1^N$, i.e.\ the distances between points are given by the $1$-metric.

The magnitude of a single cube can be computed quickly using the product formula for magnitude and Examples~\ref{example:interval-magnitude}, \ref{ex:weight-interval}.
\begin{example}\label{example:cube-weight-measure}
For any two metric spaces $M'$, $M''$ with the well-defined magnitude we have $\mg(M' \times_1 M'') = \mg(M') \mg(M'')$~\cite[Proposition 2.3.6]{leinster2010} where $\times_1$ denotes the $1$-product of metric spaces. Likewise, if $M'$, $M''$ have a weight measure, then the weight measure of $M' \times_1 M''$ is given as the product of the weight measures of $M'$ and~$M''$~\cite[Proposition 5.3.7.]{LM17}. Hence, for $p \in \ell_1^N$ and $r \in \RR_{> 0}$ we have $\mg\big(\cc_p(r)\big) = (1 + r)^N$ and the weight measure of the cube $\cc_p(r)$ is
\begin{align*}
\omega_{\cc_p(r)} \ &= \ \tfrac{1}{2^N}\!\!\prod_{k \in \NN_{\leq N}}\!\!\big(\delta_{p_k - r} + \lambda_{\intcc{p_k - r}{p_k + r}}^1 + \delta_{p_k + r}\big) \ = \ \tfrac{1}{2^N}\!\!\sum_{f \in \mathrm{faces}(\cc_p(r))}\!\!\lambda_f^{\dim(f)} \ = \\
&= \ \tfrac{1}{2^N}\!\!\prod_{k \in \NN_{\leq N}}\!\!\big(\lambda_{\intcc{p_k - r}{p_k + r}^{(0)}}^0 + \lambda_{\intcc{p_k - r}{p_k + r}^{(1)}}^1\big) \ = \ \tfrac{1}{2^N}\!\!\sum_{D \in \NN_{\leq N}}\!\!\lambda_{\cc_p(r)^{(D)}}^D.
\end{align*}
\end{example}

For any $F \in \fsub(\ell_1^N)$, let $\cc_F(r) := \bigcup_{p \in F} \cc_p(r)$. Clearly, $\cc_F(r)$ is compact, and if $F$ is non-empty, so is~$\cc_F(r)$.

\begin{corollary}\label{corollary:continuity-of-magnitude-with-cubes}
For any $N \in \NN$ and $F \in \ifsub(\ell_1^N)$, if
\[\lim_{r \decr 0} \mg\big(\cc_F(r)\big) = \mg(F),\]
then $\mg[\icsub(\ell_1^N)]$ is continuous at~$F$.
\end{corollary}

\begin{proof}
For any $r \in \RR_{> 0}$ we have $F \subseteq \cth[\ell_1^N]{F}{r} \subseteq \cc_F(r)$, so by inclusion-monotonicity of magnitude
\[\mg(F) \leq \mg\big(\cth[\ell_1^N]{F}{r}\big) \leq \mg\big(\cc_F(r)\big).\]
Hence if $\lim_{r \decr 0} \mg\big(\cc_F(r)\big) = \mg(F)$, then also $\lim_{r \decr 0} \mg\big(\cth[\ell_1^N]{F}{r}\big) = \mg(F)$, so the claim follows from Lemma~\ref{lemma:magnitude-point-continuity-characterization-in-tractable-spaces}.
\end{proof}

When the projections of the cubes onto each coordinate axis are pairwise disjoint, it turns out that the magnitude of a finite union of such cubes admits a relatively simple explicit formula (Theorem~\ref{theorem:cubes-weight-measure}). This allows us to use the above result to prove continuity of $\mg[\icsub(\ell_1^N)]$ at ``almost all'' finite subspaces of~$\ell_1^N$, namely those around which small enough cubes have pairwise disjoint projections. We call such subspaces \df{skew} subspaces.

\begin{definition}
For $S \subseteq \RR^N$ define the \df{skewness} of $S$ by
\[\skw(S) := \inf\set{|a_k - b_k|}{a, b \in S \land a \neq b \land k \in \intcc[\NN]{1}{N}},\]
i.e.\ skewness is infimum of the differences between coordinates of distinct points in~$S$. Hence, if $S$ has at least two points, then $\skw(S) \in \RR_{\geq 0}$, whereas by convention, if $S$ is empty or a singleton, we take $\skw(S) := \infty$.

A subset $S \subseteq \RR^N$ is called \df{skew} when $\skw(S) > 0$.
\end{definition}

\begin{proposition}\label{proposition:skew-implies-discrete}
A skew subset of~$\RR^N$ is necessarily discrete (in the inherited topology). In particular, a compact skew set is finite.
\end{proposition}

\begin{proof}
If $S \subseteq \RR^N$ is skew, then for every $p \in S$ the open cube (open $\infty$-ball) with the center in~$p$ and radius $\skw(S)$ (or smaller) does not contain any points in~$S$ other than~$p$, i.e.\ every point of~$S$ is isolated.
\end{proof}

\begin{proposition}
Let $F \in \fsub(\RR^N)$.
\begin{enumerate}
\item
If $\skw(F) \in \RR$ (i.e.\ if $\card{F} \geq 2$), then
\[\skw(F) = \min\set{|a_k - b_k|}{a, b \in F \land a \neq b \land k \in \intcc[\NN]{1}{N}}\]
(that is, the infimum in the definition of skewness is in fact the minimum).
\item
$F$ is skew if and only if, for every $k \in \intcc[\NN]{1}{N}$, the map $\rstr{\prj_k}_F$ (the restriction of the $k$-th projection onto~$F$) is injective.
\end{enumerate}
\end{proposition}

\begin{proof}
Follows from the definition.
\end{proof}

Note that if $F \in \ifsub(\ell_1^N)$ is skew, then for every $r \in \intoo{0}{\frac{\skw(F)}{2}}$ the cubes in~$\cc_F(r)$ have pairwise disjoint projections on axes. The goal is now to provide an explicit formula for $\mg(\cc_F(r))$ (Theorem~\ref{theorem:cubes-weight-measure}).

\section{Weight Measure of a Union of Cubes}

In order to apply Corollary~\ref{corollary:continuity-of-magnitude-with-cubes}, we first derive a formula for the magnitude of a finite union of cubes. To obtain this, we compute the corresponding weight measure. In this section, we consider finite unions of equal-sized cubes in~$\ell_1^N$ whose coordinate projections are pairwise disjoint. More precisely, we prove the following theorem.

\begin{theorem}\label{theorem:cubes-weight-measure}
Let $N \in \NN$, let $F \in \ifsub(\ell_1^N)$ be skew, and denote $m := \card{F}$. Then, for every $r \in \intoo{0}{\frac{\skw(F)}{2}}$, the union of cubes $\cc_F(r)$ has a (unique) weight measure. Denoting this measure by~$\omega_{\cc_F(r)}$, it is given by
\[\omega_{\cc_F(r)} \ = \ \Big(\tfrac{1}{2^N}\!\!\sum_{D \in \NN_{\leq N}}\!\!\lambda_{\cc_F(r)^{(D)}}^D\Big) \ \ - \!\!\sum_{\substack{p \in F \\ s \in \set{-1, 1}^N}}\!\!\!\!\alpha_{p, s}(r)\,\delta_{p + r s}\]
where the $2^N m$ coefficients $\alpha_{p, s}(r) \in \RR$ are uniquely determined by the system of $2^N m$ linear equations
\[\Bigg(\!\sum_{\substack{p \in F \\ s \in \set{-1, 1}^N}}\!\!\!\!\alpha_{p, s}(r)\,e^{-d_1(p + r s, q + r t)} \ \ = \sum_{p \in F \setminus \set{q}}\!\!\!\!e^{-d_1(\cc_p(r), q + r t)}\ \Bigg)_{q \in F,\,t \in \set{-1, 1}^N}.\]
It follows that $\displaystyle{\mg\big(\cc_F(r)\big) = m\,(1 + r)^N - \!\!\!\!\sum_{\substack{p \in F \\ s \in \set{-1, 1}^N}}\!\!\!\!\alpha_{p, s}(r)}$.
\end{theorem}

\begin{remark}
The weight measure of~$\cc_F(r)$ is the same as the weight measures of the individual cubes (compare with Example~\ref{example:cube-weight-measure}), except possibly in the vertices of cubes.
\end{remark}

We prove this theorem through a series of lemmas.

\begin{lemma}\label{lemma:integral-of-exponential-distance}
For all $a, b \in \RR$, $r \in \RR_{\geq 0}$ we have
\[\int_{b-r}^{b+r} e^{-|x - a|} \,dx =
\begin{cases}
2 e^{-|b - a|} \sinh(r) & \text{if $|b - a| \geq r$}, \\
2 \big(1 - e^{-r} \cosh(b - a)\big) & \text{if $|b - a| \leq r$}.
\end{cases}\]
\end{lemma}

\begin{proof}
We calculate the given integral in cases which cover all possibilities.
\begin{itemize}
\item{Case $b - a \geq r$:}
\[\int_{b-r}^{b+r} e^{-|x - a|} \,dx = \int_{b-r}^{b+r} e^{-(x - a)} \,dx = \Big[-e^{-(x - a)}\Big]_{b-r}^{b+r} =\]
\[= -e^{-(b + r - a)} + e^{-(b - r - a)} = e^{-(b - a)} \big(e^r - e^{-r}\big) = 2 e^{-|b - a|} \sinh(r)\]
\item{Case $b - a \leq -r$:}
\[\int_{b-r}^{b+r} e^{-|x - a|} \,dx = \int_{b-r}^{b+r} e^{x - a} \,dx = \Big[e^{x - a}\Big]_{b-r}^{b+r} =\]
\[= e^{b + r - a} - e^{b - r - a} = e^{b - a} \big(e^r - e^{-r}\big) = 2 e^{-|b - a|} \sinh(r)\]
\item{Case $-r \leq b - a \leq r$:}
\[\int_{b-r}^{b+r} e^{-|x - a|} \,dx = \int_{b-r}^{a} e^{x - a} \,dx + \int_{a}^{b+r} e^{-(x - a)} \,dx =\]
\[= \Big[e^{x - a}\Big]_{b-r}^{a} + \Big[-e^{-(x - a)}\Big]_{a}^{b+r} = 1 - e^{b - r - a} - e^{-(b + r - a)} + 1 =\]
\[= 2 - e^{-r} \big(e^{b - a} + e^{-(b - a)}\big) = 2 \big(1 - e^{-r} \cosh(b - a)\big)\]
\end{itemize}
\end{proof}

\begin{lemma}\label{leading-term-cube-weight-measure}
For all $N \in \NN$, skew $F \in \ifsub(\ell_1^N)$, $r \in \intoo{0}{\frac{\skw(F)}{2}}$, $p, q \in F$ and $t \in \intcc{-1}{1}^N$ we have
\[\tfrac{1}{2^N} \int_{\cc_p(r)} e^{-d_1(x, q + r t)} \,d\big(\lambda_{\cc_F(r)^{(0)}}^0 + \ldots + \lambda_{\cc_F(r)^{(N)}}^N\big)(x) = e^{-d_1(\cc_p(r), q + r t)}.\]
\end{lemma}

\begin{proof}
We have $\cc_p(r) = \prod_{k \in \intcc[\NN]{1}{N}} \intcc{p_k - r}{p_k + r}$, so
\begin{gather*}
\sum_{D \in \NN_{\leq N}}\!\!\lambda_{\cc_p(r)^{(D)}}^D = \prod_{k \in \intcc[\NN]{1}{N}}\!\!\Big(\lambda_{\intcc{p_k - r}{p_k + r}^{(0)}}^0 + \lambda_{\intcc{p_k - r}{p_k + r}^{(1)}}^1\Big) = \\
= \prod_{k \in \intcc[\NN]{1}{N}}\!\!\Big(\delta_{p_k - r} + \delta_{p_k + r} + \lambda_{\intcc{p_k - r}{p_k + r}}^1\Big).
\end{gather*}
Hence,
\begin{gather*}
\int_{\cc_p(r)} e^{-d_1(x, q + r t)} \,d\big(\lambda_{\cc_F(r)^{(0)}}^0 + \ldots + \lambda_{\cc_F(r)^{(N)}}^N\big)(x) = \\
= \int_{\cc_p(r)} e^{-d_1(x, q + r t)} \,d\big(\lambda_{\cc_p(r)^{(0)}}^0 + \ldots + \lambda_{\cc_p(r)^{(N)}}^N\big)(x) = \\
=\!\!\prod_{k \in \intcc[\NN]{1}{N}}\!\!\Big(e^{-|p_k - r - (q_k + r t_k)|} + e^{-|p_k + r - (q_k + r t_k)|} + \int_{p_k - r}^{p_k + r} e^{-|x_k - (q_k + r t_k)|} \,dx_k\Big).
\end{gather*}
We now split the proof into two cases.
\begin{itemize}
\item{Case $p \neq q$:}

Because $F$ is skew, we have $p_k\neq q_k$ for all $k \in \intcc[\NN]{1}{N}$. If $p_k > q_k$, then also $p_k > q_k + 2r$ since by assumption $2r < \skw(F)$. In this case
\[e^{-|p_k - r - (q_k + r t_k)|} + e^{-|p_k + r - (q_k + r t_k)|} = e^{-(p_k - r - (q_k + r t_k))} + e^{-(p_k + r - (q_k + r t_k))} =\]
\[= e^{-(p_k - (q_k + r t_k))} \big(e^r + e^{-r}\big) = 2 e^{-(p_k - (q_k + r t_k))} \cosh(r).\]
Similarly, if $p_k < q_k$, we get
\[e^{-|p_k - r - (q_k + r t_k)|} + e^{-|p_k + r - (q_k + r t_k)|} = 2 e^{p_k - (q_k + r t_k)} \cosh(r),\]
so either way
\[e^{-|p_k - r - (q_k + r t_k)|} + e^{-|p_k + r - (q_k + r t_k)|} = 2 e^{-|p_k - (q_k + r t_k)|} \cosh(r).\]
Then, using Lemma~\ref{lemma:integral-of-exponential-distance},
\[\int_{\cc_p(r)} e^{-d_1(x, q + r t)} \,d\big(\lambda_{\cc_F(r)^{(0)}}^0 + \ldots + \lambda_{\cc_F(r)^{(N)}}^N\big)(x) =\]
\[=\!\!\prod_{k \in \intcc[\NN]{1}{N}}\!\!\Big(2 e^{-|p_k - (q_k + r t_k)|} \cosh(r) + 2 e^{-|p_k - (q_k + r t_k)|} \sinh(r)\Big) =\]
\[=\!\!\prod_{k \in \intcc[\NN]{1}{N}}\!\!2 e^{-|p_k - (q_k + r t_k)|}\,e^r = 2^N e^{N r}\!\!\prod_{k \in \intcc[\NN]{1}{N}}\!\!e^{-|p_k - (q_k + r t_k)|} =\]
\[= 2^N e^{N r} e^{-\sum_{k \in \intcc[\NN]{1}{N}} |p_k - (q_k + r t_k)|} = 2^N e^{N r} e^{-d_1(p, q + r t)} =\]
\[= 2^N e^{-(d_1(p, q + r t) - N r)} = 2^N e^{-d_1(\cc_p(r), a_j + r t)}.\]
\item{Case $p = q$:}

We need to prove
\[\tfrac{1}{2^N} \int_{\cc_p(r)} e^{-d_1(x, q + r t)} \,d\big(\lambda_{\cc_F(r)^{(0)}}^0 + \ldots + \lambda_{\cc_F(r)^{(N)}}^N\big)(x) = e^{-d_1(\cc_p(r), q + r t)} = e^{-0} = 1.\]
This actually follows immediately from the fact that the restriction of $\tfrac{1}{2^N} \big(\lambda_{\cc_F(r)^{(0)}}^0 + \ldots + \lambda_{\cc_F(r)^{(N)}}^N\big)$ to~$\cc_p(r)$ is the weight measure of~$\cc_p(r)$, but we also give an explicit calculation below (again using Lemma~\ref{lemma:integral-of-exponential-distance}).
\[\int_{\cc_p(r)} e^{-d_1(x, q + r t)} \,d\big(\lambda_{\cc_F(r)^{(0)}}^0 + \ldots + \lambda_{\cc_F(r)^{(N)}}^N\big)(x) =\]
\[=\!\!\prod_{k \in \intcc[\NN]{1}{N}}\!\!\Big(e^{-|p_k - r - (q_k + r t_k)|} + e^{-|p_k + r - (q_k + r t_k)|} + \int_{p_k - r}^{p_k + r} e^{-|x_k - (q_k + r t_k)|} \,dx_k\Big) =\]
\[=\!\!\prod_{k \in \intcc[\NN]{1}{N}}\!\!\Big(e^{-r(1 + t_k)} + e^{-r(1 - t_k)} + 2 - 2 e^{-r} \cosh(r t_k)\Big) =\!\!\prod_{k \in \intcc[\NN]{1}{N}}\!\!\!\!2\,= 2^N\]
\end{itemize}
\end{proof}

\begin{definition}\label{definition:corners}
Given $N \in \NN$, $F \in \ifsub(\ell_1^N)$, $q \in F$ and $u \in \set{-1, 1}^N$, we define the \df{$u$-corner of~$F$ at~$q$} to be the set
\[\cn{q, u} := \set{p \in F}{\all{k \in \intcc[\NN]{1}{N}}{\sgn(p_k - q_k) = u_k}}.\]
That is, $\cn{q, u}$ contains all those points of~$F$ which, if we imagine the origin in~$q$, lie in the open orthant in the direction~$u$.
\end{definition}

\begin{example}
Consider $F = \set{a, b, c, d} \subseteq \ell_1^2$, as given by the following picture.
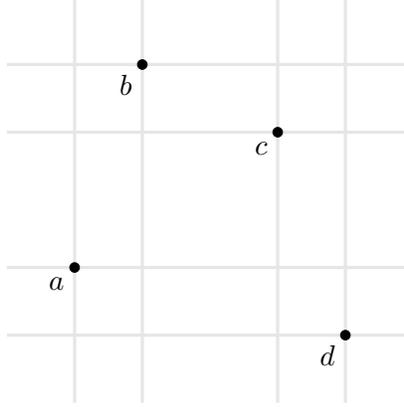
\begin{center}
\begin{tikzpicture}[scale = 0.9]
\draw[black!10, very thick] (-1, 0) -- (5, 0);
\draw[black!10, very thick] (-1, 3) -- (5, 3);
\draw[black!10, very thick] (-1, 2) -- (5, 2);
\draw[black!10, very thick] (-1, -1) -- (5, -1);
\draw[black!10, very thick] (0, -2) -- (0, 4);
\draw[black!10, very thick] (1, -2) -- (1, 4);
\draw[black!10, very thick] (3, -2) -- (3, 4);
\draw[black!10, very thick] (4, -2) -- (4, 4);
\filldraw (0, 0) circle (2pt) node [below left] {$a$};
\filldraw (1, 3) circle (2pt) node [below left] {$b$};
\filldraw (3, 2) circle (2pt) node [below left] {$c$};
\filldraw (4, -1) circle (2pt) node [below left] {$d$};
\end{tikzpicture}
\end{center}
The following is the list of all corners in this example.
\begin{align*}
&\cn{a, (-1, -1)} = \emptyset & &\cn{b, (-1, -1)} = \set{a} & &\cn{c, (-1, -1)} = \set{a} & &\cn{d, (-1, -1)} = \emptyset \\
&\cn{a, (-1, 1)} = \emptyset & &\cn{b, (-1, 1)} = \emptyset & &\cn{c, (-1, 1)} = \set{b} & &\cn{d, (-1, 1)} = \set{a, b, c} \\
&\cn{a, (1, -1)} = \set{d} & &\cn{b, (1, -1)} = \set{c, d} & &\cn{c, (1, -1)} = \set{d} & &\cn{d, (1, -1)} = \emptyset \\
&\cn{a, (1, 1)} = \set{b, c} & &\cn{b, (1, 1)} = \emptyset & &\cn{c, (1, 1)} = \emptyset & &\cn{d, (1, 1)} = \emptyset
\end{align*}
\end{example}

\begin{lemma}\label{lemma:vertex-directions}
Let $N \in \NN$ and let $F \in \ifsub(\ell_1^N)$ be skew.
\begin{enumerate}
\item\label{lemma:vertex-directions:disjoint-union}
For every $q \in F$, we have
\[\bigcup_{u \in \set{-1, 1}^N}\!\!\!\!\cn{q, u} \ = \ F \setminus \set{q},\]
and this is a disjoint union.
\item\label{lemma:vertex-directions:distance-split}
For all $r \in \intoo{0}{\frac{\skw(F)}{2}}$, $p, q \in F$, $s, t \in \intcc{-1}{1}^N$ and $u \in \set{-1, 1}^N$, if $p \in \cn{q, u}$, then
\[d_1(p + r s, q + r t) = d_1(p + r s, q + r u) + d_1(r u, r t)\]
and
\[d_1\big(\cc_p(r), q + r t\big) = d_1\big(\cc_p(r), q + r u\big) + d_1(r u, r t).\]
\end{enumerate}
\end{lemma}

\begin{proof}
\
\begin{enumerate}
\item
That sets $\cn{q, u}$ (for a fixed~$q$) are pairwise disjoint and contained in~$F \setminus \set{q}$ is obvious from the definition. That they cover $F \setminus \set{q}$ follows from the fact that $F$ is skew.
\item
In the calculation below, we use the following facts:
\begin{itemize}
\item
$|x| = \sgn(x) \cdot x$ for all $x \in \RR$,
\item
for any $k \in \intcc[\NN]{1}{N}$, the statement $\sgn(p_k - q_k) = u_k$ implies the statements $\sgn(p_k + r s_k - (q_k + r t_k)) = u_k$ and $\sgn(p_k + r s_k - (q_k + r u_k)) = u_k$ since by assumption $r \in \intoo{0}{\frac{\skw(F)}{2}}$,
\item
$u_k^2 - u_k t_k = 1 - u_k t_k \geq 0$, so $u_k^2 - u_k t_k = |u_k^2 - u_k t_k| = |u_k(u_k - t_k)| = |u_k - t_k|$.
\end{itemize}
\begin{gather*}
d_1(p + r s, q + r t) =\!\!\sum_{k \in \intcc[\NN]{1}{N}}\!\!\big|p_k + r s_k - (q_k + r t_k)\big| = \\
=\!\!\sum_{k \in \intcc[\NN]{1}{N}}\!\!u_k\big(p_k + r s_k - (q_k + r t_k)\big) = \\
=\!\!\sum_{k \in \intcc[\NN]{1}{N}}\!\!u_k\big(p_k + r s_k - (q_k + r u_k)\big) + r\!\!\sum_{k \in \intcc[\NN]{1}{N}}\!\!(u_k^2 - u_k t_k) = \\
=\!\!\sum_{k \in \intcc[\NN]{1}{N}}\!\!\big|p_k + r s_k - (q_k + r u_k)\big| + r\!\!\sum_{k \in \intcc[\NN]{1}{N}}\!\!|u_k - t_k| = \\
= d_1(p + r s, q + r u) + d_1(r u, r t)
\end{gather*}
This shows the first claimed equality; to get the second, apply the above for $s = 0$ and calculate:
\begin{gather*}
d_1\big(\cc_p(r), q + r t\big) = d_1(p, q + r t) - N r = d_1(p, q + r u) + d_1(r u, r t) - N r = \\
= d_1\big(\cc_p(r), q + r u\big) + d_1(r u, r t).
\end{gather*}
\end{enumerate}
\end{proof}

\begin{definition}
For any $N \in \NN$, $F \in \ifsub(\ell_1^N)$ and $r \in \RR_{\geq 0}$, we define the following two systems of linear equations for unknowns~$x_{p, s}$.
\begin{samepage}
\begin{center}
\underline{Vertex~System}:
\end{center}
\[\Bigg(\!\sum_{\substack{p \in F \\ s \in \set{-1, 1}^N}}\!\!\!\!x_{p, s}\,e^{-d_1(p + r s, q + r t)} \ \ = \sum_{p \in F \setminus \set{q}}\!\!\!\!e^{-d_1(\cc_p(r), q + r t)}\ \Bigg)_{q \in F,\,t \in \set{-1, 1}^N}\]
\end{samepage}
\begin{samepage}
\begin{center}
\underline{Corner~System}:
\end{center}
\[\Bigg(x_{q, u} \ +\!\!\!\!\sum_{\substack{p \in \cn{q, u} \\ s \in \set{-1, 1}^N}}\!\!\!\!x_{p, s}\,e^{-d_1(p + r s, q + r u)} \ \ = \sum_{p \in \cn{q, u}}\!\!e^{-d_1(\cc_p(r), q + r u)}\ \Bigg)_{q \in F,\,u \in \set{-1, 1}^N},\]
\end{samepage}

Both systems have $2^N \cdot \card{F}$ equations and the same number of unknowns.
\end{definition}

\begin{example}\label{example:systems}
Take $F = \set{(0, 0), (4, 8), (7, 3)} \subseteq \ell_1^2$. Its Vertex System (in matrix form) looks like this:
{\tiny\setlength{\arraycolsep}{1pt}
\[
\begin{pmatrix}
1 & e^{-2r} & e^{-2r} & e^{-4r} & e^{-12} & e^{-(12 + 2r)} & e^{-(12 + 2r)} & e^{-(12 + 4r)} & e^{-10} & e^{-(10 + 2r)} & e^{-(10 + 2r)} & e^{-(10 + 4r)} \\
e^{-2r} & 1 & e^{-4r} & e^{-2r} & e^{-(12 - 2r)} & e^{-12} & e^{-12} & e^{-(12 + 2r)} & e^{-(10 - 2r)} & e^{-10} & e^{-10} & e^{-(10 + 2r)} \\
e^{-2r} & e^{-4r} & 1 & e^{-2r} & e^{-(12 - 2r)} & e^{-12} & e^{-12} & e^{-(12 + 2r)} & e^{-(10 - 2r)} & e^{-10} & e^{-10} & e^{-(10 + 2r)} \\
e^{-4r} & e^{-2r} & e^{-2r} & 1 & e^{-(12 - 4r)} & e^{-(12 - 2r)} & e^{-(12 - 2r)} & e^{-12} & e^{-(10 - 4r)} & e^{-(10 - 2r)} & e^{-(10 - 2r)} & e^{-10} \\

e^{-12} & e^{-(12 - 2r)} & e^{-(12 - 2r)} & e^{-(12 - 4r)} & 1 & e^{-2r} & e^{-2r} & e^{-4r} & e^{-8} & e^{-(8 - 2r)} & e^{-(8 + 2r)} & e^{-8} \\
e^{-(12 + 2r)} & e^{-12} & e^{-12} & e^{-(12 - 2r)} & e^{-2r} & 1 & e^{-4r} & e^{-2r} & e^{-(8 + 2r)} & e^{-8} & e^{-(8 + 4r)} & e^{-(8 + 2r)} \\
e^{-(12 + 2r)} & e^{-12} & e^{-12} & e^{-(12 - 2r)} & e^{-2r} & e^{-4r} & 1 & e^{-2r} & e^{-(8 - 2r)} & e^{-(8 - 4r)} & e^{-8} & e^{-(8 - 2r)} \\
e^{-(12 + 4r)} & e^{-(12 + 2r)} & e^{-(12 + 2r)} & e^{-12} & e^{-4r} & e^{-2r} & e^{-2r} & 1 & e^{-8} & e^{-(8 - 2r)} & e^{-(8 + 2r)} & e^{-8} \\

e^{-10} & e^{-(10 - 2r)} & e^{-(10 - 2r)} & e^{-(10 - 4r)} & e^{-8} & e^{-(8 + 2r)} & e^{-(8 - 2r)} & e^{-8} & 1 & e^{-2r} & e^{-2r} & e^{-4r} \\
e^{-(10 + 2r)} & e^{-10} & e^{-10} & e^{-(10 - 2r)} & e^{-(8 - 2r)} & e^{-8} & e^{-(8 - 4r)} & e^{-(8 - 2r)} & e^{-2r} & 1 & e^{-4r} & e^{-2r} \\
e^{-(10 + 2r)} & e^{-10} & e^{-10} & e^{-(10 - 2r)} & e^{-(8 + 2r)} & e^{-(8 + 4r)} & e^{-8} & e^{-(8 + 2r)} & e^{-2r} & e^{-4r} & 1 & e^{-2r} \\
e^{-(10 + 4r)} & e^{-(10 + 2r)} & e^{-(10 + 2r)} & e^{-10} & e^{-8} & e^{-(8 + 2r)} & e^{-(8 - 2r)} & e^{-8} & e^{-4r} & e^{-2r} & e^{-2r} & 1
\end{pmatrix}
\cdot
\]
\[
\cdot
\begin{pmatrix}
x_{(0, 0), (-1, -1)} \\
x_{(0, 0), (-1, 1)} \\
x_{(0, 0), (1, -1)} \\
x_{(0, 0), (1, 1)} \\

x_{(4, 8), (-1, -1)} \\
x_{(4, 8), (-1, 1)} \\
x_{(4, 8), (1, -1)} \\
x_{(4, 8), (1, 1)} \\

x_{(7, 3), (-1, -1)} \\
x_{(7, 3), (-1, 1)} \\
x_{(7, 3), (1, -1)} \\
x_{(7, 3), (1, 1)}
\end{pmatrix}
=
\begin{pmatrix}
e^{-12} + e^{-10} \\
e^{-(12 - 2r)} + e^{-(10 - 2r)} \\
e^{-(12 - 2r)} + e^{-(10 - 2r)} \\
e^{-(12 - 4r)} + e^{-(10 - 4r)} \\

e^{-(12 - 4r)} + e^{-(8 - 2r)} \\
e^{-(12 - 2r)} + e^{-8} \\
e^{-(12 - 2r)} + e^{-(8 - 4r)} \\
e^{-12} + e^{-(8 - 2r)} \\

e^{-(10 - 4r)} + e^{-(8 - 2r)} \\
e^{-(10 - 2r)} + e^{-(8 - 4r)} \\
e^{-(10 - 2r)} + e^{-8} \\
e^{-10} + e^{-(8 - 2r)}
\end{pmatrix}
\]
}

We have $\cn{(0, 0), (1, 1)} = \set{(4, 8), (7, 3)}$, $\cn{(4, 8), (-1, -1)} = \set{(0, 0)}$, $\cn{(4, 8), (1, -1)} = \set{(7, 3)}$, $\cn{(7, 3), (-1, -1)} = \set{(0, 0)}$, $\cn{(7, 3), (-1, 1)} = \set{(4, 8)}$, and the other corners are empty. Hence, the Corner System of the given~$F$ looks as follows:
{\tiny\setlength{\arraycolsep}{1pt}
\[
\begin{pmatrix}
1 & 0 & 0 & 0 & 0 & 0 & 0 & 0 & 0 & 0 & 0 & 0 \\
0 & 1 & 0 & 0 & 0 & 0 & 0 & 0 & 0 & 0 & 0 & 0 \\
0 & 0 & 1 & 0 & 0 & 0 & 0 & 0 & 0 & 0 & 0 & 0 \\
0 & 0 & 0 & 1 & e^{-(12 - 4r)} & e^{-(12 - 2r)} & e^{-(12 - 2r)} & e^{-12} & e^{-(10 - 4r)} & e^{-(10 - 2r)} & e^{-(10 - 2r)} & e^{-10} \\

e^{-12} & e^{-(12 - 2r)} & e^{-(12 - 2r)} & e^{-(12 - 4r)} & 1 & 0 & 0 & 0 & 0 & 0 & 0 & 0 \\
0 & 0 & 0 & 0 & 0 & 1 & 0 & 0 & 0 & 0 & 0 & 0 \\
0 & 0 & 0 & 0 & 0 & 0 & 1 & 0 & e^{-(8 - 2r)} & e^{-(8 - 4r)} & e^{-8} & e^{-(8 - 2r)} \\
0 & 0 & 0 & 0 & 0 & 0 & 0 & 1 & 0 & 0 & 0 & 0 \\

e^{-10} & e^{-(10 - 2r)} & e^{-(10 - 2r)} & e^{-(10 - 4r)} & 0 & 0 & 0 & 0 & 1 & 0 & 0 & 0 \\
0 & 0 & 0 & 0 & e^{-(8 - 2r)} & e^{-8} & e^{-(8 - 4r)} & e^{-(8 - 2r)} & 0 & 1 & 0 & 0 \\
0 & 0 & 0 & 0 & 0 & 0 & 0 & 0 & 0 & 0 & 1 & 0 \\
0 & 0 & 0 & 0 & 0 & 0 & 0 & 0 & 0 & 0 & 0 & 1
\end{pmatrix}
\cdot
\]
\[
\cdot
\begin{pmatrix}
x_{(0, 0), (-1, -1)} \\
x_{(0, 0), (-1, 1)} \\
x_{(0, 0), (1, -1)} \\
x_{(0, 0), (1, 1)} \\

x_{(4, 8), (-1, -1)} \\
x_{(4, 8), (-1, 1)} \\
x_{(4, 8), (1, -1)} \\
x_{(4, 8), (1, 1)} \\

x_{(7, 3), (-1, -1)} \\
x_{(7, 3), (-1, 1)} \\
x_{(7, 3), (1, -1)} \\
x_{(7, 3), (1, 1)}
\end{pmatrix}
=
\begin{pmatrix}
0 \\
0 \\
0 \\
e^{-(12 - 4r)} + e^{-(10 - 4r)} \\

e^{-(12 - 4r)} \\
0 \\
e^{-(8 - 4r)} \\
0 \\

e^{-(10 - 4r)} \\
e^{-(8 - 4r)} \\
0 \\
0
\end{pmatrix}
\]
}
\end{example}

\begin{lemma}\label{lemma:coefficients}
Let $N \in \NN$ and let $F \in \ifsub(\ell_1^N)$ be skew. Then, for every $r \in \intoo{0}{\frac{\skw(F)}{2}}$, the following statements hold.
\begin{enumerate}
\item
The coefficient matrix of the Vertex System is $Z_{\cc_F(r)^{(0)}}$ (the similarity matrix of the set of vertices in~$\cc_F(r)$)\footnote{The reason for the name ``Vertex System'' is not just that its coefficient matrix is the similarity matrix of the vertices of~$\cc_F(r)$, but also that $\alpha_{p, s}(r)$ in Theorem~\ref{theorem:cubes-weight-measure} being defined as solutions of this system amounts to the defining condition for a weight measure of~$\cc_F(r)$ in each vertex of~$\cc_F(r)$ (see the proof of Theorem~\ref{theorem:cubes-weight-measure} below). In other words, Theorem~\ref{theorem:cubes-weight-measure} essentially states: if we choose $\alpha_{p, s}(r)$ in the formula for $\omega_{\cc_F(r)}$ so that the condition for a weight measure is satisfied in the vertices of~$\cc_F(r)$, then it is satisfied everywhere on~$\cc_F(r)$.} which is invertible. Hence, the system has a unique solution.
\item
The Vertex System and the Corner System are equivalent (i.e.\ they have the same solutions). Hence, the Corner System also has a unique solution.\footnote{The fact that $\alpha_{p, s}(r)$ in Theorem~\ref{theorem:cubes-weight-measure} can equivalently be given by both the Vertex System and the Corner System is crucial for the results of this paper. It is the central part of both the argument that the given $\omega_{\cc_F(r)}$ is the weight measure of~$\cc_F(r)$, and the argument that magnitude is continuous at~$F$.}
\end{enumerate}
\end{lemma}

\begin{proof}
\
\begin{enumerate}
\item
Since by assumption $r \in \intoo{0}{\frac{\skw(F)}{2}}$, the cubes $\cc_p(r)$ are pairwise disjoint, so $\cc_F(r)$ has $2^N m$ different vertices. We observe that the coefficient matrix of Vertex System is
\[\big(e^{-d_1(p + r s, q + r t)}\big)_{p, q \in F, s, t \in \set{-1, 1}^N},\]
so indeed precisely the similarity matrix of vertices~$\cc_F(r)^{(0)}$. Since the space~$\ell_1^N$ is positive definite, this matrix is positive definite, in particular invertible. Thus the given system has a unique solution.
\item
Assume that $x_{p, s}$ are solutions of the Corner System of equations, which we shorten to $\big(L_{q, u} = R_{q, u}\big)_{q \in F,\,u \in \set{-1, 1}^N}$. Then, for every $q \in F$ and $t \in \set{-1, 1}^N$ we have
\[\sum_{u \in \set{-1, 1}^N}\!\!\!\!e^{-d_1(r u, r t)} L_{q, u} \ =\!\!\sum_{u \in \set{-1, 1}^N}\!\!\!\!e^{-d_1(r u, r t)} R_{q, u}.\]
Use Lemma~\ref{lemma:vertex-directions} to evaluate both sides. Left-hand side:
\begin{gather*}
\sum_{u \in \set{-1, 1}^N}\!\!\!\!e^{-d_1(r u, r t)} L_{q, u} = \\
= \sum_{u \in \set{-1, 1}^N}\!\!\!\!e^{-d_1(r u, r t)} \bigg(x_{q, u} \ +\!\!\!\!\sum_{\substack{p \in \cn{q, u} \\ s \in \set{-1, 1}^N}}\!\!\!\!x_{p, s}\,e^{-d_1(p + r s, q + r u)}\bigg) = \\
= \Big(\sum_{u \in \set{-1, 1}^N}\!\!\!\!x_{q, u}\,e^{-d_1(q + r u, q + r t)}\Big) \ +\!\!\!\!\sum_{u \in \set{-1, 1}^N}\!\!\sum_{\substack{p \in \cn{q, u} \\ s \in \set{-1, 1}^N}}\!\!\!\!x_{p, s}\,e^{-d_1(p + r s, q + r t)} = \\
= \sum_{\substack{p \in F \\ s \in \set{-1, 1}^N}}\!\!\!\!x_{p, s}\,e^{-d_1(p + r s, q + r t)}.
\end{gather*}
Right-hand side:
\begin{gather*}
\sum_{u \in \set{-1, 1}^N}\!\!\!\!e^{-d_1(r u, r t)} R_{q, u} \ =\!\!\sum_{u \in \set{-1, 1}^N}\!\!\bigg(e^{-d_1(r u, r t)} \sum_{p \in \cn{q, u}}\!\!e^{-d_1(\cc_p(r), q + r u)}\bigg) = \\
= \sum_{u \in \set{-1, 1}^N} \sum_{p \in \cn{q, u}}\!\!e^{-d_1(\cc_p(r), q + r t)} \ = \sum_{p \in F \setminus \set{q}}\!\!\!\!e^{-d_1(\cc_p(r), q + r t)}.
\end{gather*}
We conclude that $x_{p, s}$ also solve the Vertex System.

Conversely, assume that
\[\sum_{\substack{p \in F \\ s \in \set{-1, 1}^N}}\!\!\!\!x_{p, s}\,e^{-d_1(p + r s, q + r t)} \ \ = \sum_{p \in F \setminus \set{q}}\!\!\!\!e^{-d_1(\cc_p(r), q + r t)}\]
for every $q \in F$ and $t \in \set{-1, 1}^N$. By the above calculation this is the same as
\[\sum_{u \in \set{-1, 1}^N}\!\!\!\!e^{-d_1(r u, r t)} L_{q, u} \ =\!\!\sum_{u \in \set{-1, 1}^N}\!\!\!\!e^{-d_1(r u, r t)} R_{q, u}.\]
In matrix form, this states
\[Z_{\cc_q(r)^{(0)}} \cdot L_q = Z_{\cc_q(r)^{(0)}} \cdot R_q\]
where $L_q = (L_{q, u})_{u \in \set{-1, 1}^N}$, $R_q = (R_{q, u})_{u \in \set{-1, 1}^N}$, and $Z_{\cc_q(r)^{(0)}}$ is the similarity matrix of the set of vertices of the cube~$\cc_q(r)$. As this matrix is positive definite, therefore invertible, we may cancel it to obtain $L_q = R_q$. In conclusion, solutions of the Vertex System are also solutions of the Corner System.
\end{enumerate}
\end{proof}

We can now prove Theorem~\ref{theorem:cubes-weight-measure}.
\begin{proof}[Proof of Theorem~\ref{theorem:cubes-weight-measure}]
By definition the coefficients~$\alpha_{i, s}(r)$ solve the Vertex System which we know has a unique solution by Lemma~\ref{lemma:coefficients}.

Let us verify that the given $\omega_{\cc_F(r)}$ is indeed a weight measure for~$\cc_F(r)$. That is, for every $q \in F$ and $t \in \intcc{-1}{1}^N$ we need to show
\[\int_{\cc_F(r)} e^{-d_1(x, q + r t)} d\omega_{\cc_F(r)}(x) = 1.\]
This is proved by the following calculation.

\noindent\makebox[\linewidth][l]{$\displaystyle
\begin{aligned}
&\int_{\cc_F(r)} e^{-d_1(x, q + r t)} d\omega_{\cc_F(r)}(x) = \\
&\quad\quad\text{(cubes in $\cc_F(r)$ are disjoint)} \\
&= \sum_{p \in F} \int_{\cc_p(r)} e^{-d_1(x, q + r t)} d\omega_{\cc_F(r)}(x) = \\
&\quad\quad\text{(definition of $\omega_{\cc_F(r)}$ and Lemma~\ref{leading-term-cube-weight-measure})} \\
&= \sum_{p \in F} \Big(e^{-d_1(\cc_p(r), q + r t)} - \sum_{s \in \set{-1, 1}^N} \alpha_{p, s}(r) e^{-d_1(p + r s, q + r t)}\Big) = \\
\end{aligned}$}

 \begin{align*}
&= \Big(\sum_{p \in F} e^{-d_1(\cc_p(r), q + r t)}\Big) -\!\!\sum_{\substack{p \in F \\ s \in \set{-1, 1}^N}}\!\!\!\! \alpha_{p, s}(r) e^{-d_1(p + r s, q + r t)} = \\
&= \Big(\sum_{p \in F} e^{-d_1(\cc_p(r), q + r t)}\Big) - \sum_{s \in \set{-1, 1}^N} \Big(\alpha_{q, s}(r) e^{-d_1(q + r s, q + r t)} + \sum_{p \in F \setminus \set{q}} \alpha_{p, s}(r) e^{-d_1(p + r s, q + r t)}\Big) = \\
&\quad\quad\text{(Lemma~\ref{lemma:vertex-directions}~\ref{lemma:vertex-directions:disjoint-union})} \\
&= \Big(\sum_{p \in F} e^{-d_1(\cc_p(r), q + r t)}\Big) - \sum_{s \in \set{-1, 1}^N} \Big(\alpha_{q, s}(r) e^{-d_1(r s, r t)} + \sum_{u \in \set{-1, 1}^N} \sum_{p \in \cn{q, u}} \alpha_{p, s}(r) e^{-d_1(p + r s, q + r t)}\Big) = \\
&\quad\quad\text{(Lemma~\ref{lemma:vertex-directions}~\ref{lemma:vertex-directions:distance-split})} \\
&= \Big(\sum_{p \in F} e^{-d_1(\cc_p(r), q + r t)}\Big) - \Big(\sum_{s \in \set{-1, 1}^N} \alpha_{q, s}(r) e^{-d_1(r s, r t)}\Big) - \sum_{u \in \set{-1, 1}^N}\!\!\sum_{\substack{p \in \cn{q, u} \\ s \in \set{-1, 1}^N}}\!\!\!\! \alpha_{p, s}(r) e^{-d_1(p + r s, q + r u)} e^{-d_1(r u, r t)} = \\
&= \Big(\sum_{p \in F} e^{-d_1(\cc_p(r), q + r t)}\Big) - \Big(\sum_{u \in \set{-1, 1}^N} \alpha_{q, u}(r) e^{-d_1(r u, r t)}\Big) - \sum_{u \in \set{-1, 1}^N}\!\!\sum_{\substack{p \in \cn{q, u} \\ s \in \set{-1, 1}^N}}\!\!\!\! \alpha_{p, s}(r) e^{-d_1(p + r s, q + r u)} e^{-d_1(r u, r t)} = \\
&= \Big(\sum_{p \in F} e^{-d_1(\cc_p(r), q + r t)}\Big) - \sum_{u \in \set{-1, 1}^N} \Big(\alpha_{q, u}(r) +\!\!\sum_{\substack{p \in \cn{q, u} \\ s \in \set{-1, 1}^N}}\!\!\!\! \alpha_{p, s}(r) e^{-d_1(p + r s, q + r u)}\Big) e^{-d_1(r u, r t)} = \\
&\quad\quad\text{(Lemma~\ref{lemma:coefficients})} \\
&= \Big(\sum_{p \in F} e^{-d_1(\cc_p(r), q + r t)}\Big) - \sum_{u \in \set{-1, 1}^N} \sum_{p \in \cn{q, u}} e^{-d_1(\cc_p(r), q + r u)} e^{-d_1(r u, r t)} = \\
&\quad\quad\text{(Lemma~\ref{lemma:vertex-directions}~\ref{lemma:vertex-directions:distance-split})} \\
&= \Big(\sum_{p \in F} e^{-d_1(\cc_p(r), q + r t)}\Big) - \sum_{u \in \set{-1, 1}^N} \sum_{p \in \cn{q, u}} e^{-d_1(\cc_p(r), q + r t)} = \\
&\quad\quad\text{(Lemma~\ref{lemma:vertex-directions}~\ref{lemma:vertex-directions:disjoint-union})} \\
&= \Big(\sum_{p \in F} e^{-d_1(\cc_p(r), q + r t)}\Big) - \sum_{p \in F \setminus \set{q}} e^{-d_1(\cc_p(r), q + r t)} = \\
&= e^{-d_1(\cc_q(r), q + r t)} = 1
\end{align*}

Observe that if we restrict $t$ to $\set{-1, 1}^N$, then the equalities
\[1 = \int_{\cc_F(r)} e^{-d_1(x, q + r t)} d\omega_{\cc_F(r)}(x) = \Big(\sum_{p \in F} e^{-d_1(\cc_p(r), q + r t)}\Big) -\!\!\sum_{\substack{p \in F \\ s \in \set{-1, 1}^N}}\!\!\!\! \alpha_{p, s}(r) e^{-d_1(p + r s, q + r t)}\]
amount to precisely the equations in the Vertex System. Hence, the weight measure for~$\cc_F(r)$ is indeed uniquely determined.

We calculate magnitude from the weight measure in the usual way.
\begin{gather*}
\mg\big(\cc_f(r)\big) = \int_{\cc_F(r)} d\omega_{\cc_F(r)}(x) = \\
= \tfrac{1}{2^N} \int_{\cc_F(r)} d\big(\lambda_{\cc_F(r)^{(0)}}^0 + \ldots + \lambda_{\cc_F(r)^{(N)}}^N\big)(x) \ \ - \!\!\sum_{\substack{p \in F \\ s \in \set{-1, 1}^N}}\!\!\!\!\Big(\alpha_{p, s}(r) \int_{\cc_F(r)}\!\!d\delta_{p + r s}(x)\Big) = \\
= \tfrac{1}{2^N} \sum_{p \in F} \int_{\cc_p(r)} d\big(\lambda_{\cc_p(r)^{(0)}}^0 + \ldots + \lambda_{\cc_p(r)^{(N)}}^N\big)(x) \ \ - \!\!\sum_{\substack{p \in F \\ s \in \set{-1, 1}^N}}\!\!\!\!\alpha_{p, s}(r) = \\
= \tfrac{1}{2^N} \sum_{p \in F} \prod_{k \in \intcc[\NN]{1}{N}} \int_{\intcc{p_k - r}{p_k + r}} d\big(\lambda_{\intcc{p_k - r}{p_k + r}^{(0)}}^0 + \lambda_{\intcc{p_k - r}{p_k + r}^{(1)}}^1\big)(x) \ \ - \!\!\sum_{\substack{p \in F \\ s \in \set{-1, 1}^N}}\!\!\!\!\alpha_{p, s}(r) = \end{gather*}
\begin{gather*}
= \tfrac{1}{2^N} \sum_{p \in F} \prod_{k \in \intcc[\NN]{1}{N}}\!\!(2 + 2r) \ \ - \!\!\sum_{\substack{p \in F \\ s \in \set{-1, 1}^N}}\!\!\!\!\alpha_{p, s}(r) = \\
= \tfrac{1}{2^N} \cdot m \cdot (2 + 2r)^N \ \ - \!\!\sum_{\substack{p \in F \\ s \in \set{-1, 1}^N}}\!\!\!\!\alpha_{p, s}(r) = \\
= m\,(1 + r)^N \ \ - \!\!\sum_{\substack{p \in F \\ s \in \set{-1, 1}^N}}\!\!\!\!\alpha_{p, s}(r)
\end{gather*}
\end{proof}

\section{Continuity of Magnitude at Skew Finite Subsets}

In this section we use the formula for the magnitude of a union of cubes from Theorem~\ref{theorem:cubes-weight-measure} to prove that magnitude is continuous at every skew finite subset in~$\ell_1^N$. Afterwards, we illustrate our results on several examples.

\begin{lemma}\label{lemma:coefficient-limit}
Let the notation and assumptions be as in Theorem~\ref{theorem:cubes-weight-measure}. Additionally, let $w = (w_p)_{p \in F}$ be the weighting of~$F$.
\begin{enumerate}
\item
The maps $\alpha_{p, s}\colon \intoo{0}{\frac{\skw(F)}{2}} \to \RR$ are continuous. Moreover, they have continuous extensions 
\[
\alpha_{p, s}\colon \intco{0}{\frac{\skw(F)}{2}} \to \RR,
\]
i.e.\ for every $p \in F$ and $s \in \set{-1, 1}^N$ there exists the limit $\lim_{r \decr 0} \alpha_{p, s}(r)$, and we have
\[\alpha_{p, s}(0) := \lim_{r \decr 0} \alpha_{p, s}(r) = \sum_{q \in \cn{p, s}} e^{-d_1(p, q)} w_q.\]
\item
Define $\sigma_p\colon \intco{0}{\frac{\skw(F)}{2}} \to \RR$ for each $p \in F$ as
\[\sigma_p(r) := \sum_{s \in \set{-1, 1}^N} \alpha_{p, s}(r).\]
Then $\sigma_p(0) = 1 - w_p$.
\item
We have
\[\sum_{\substack{p \in F \\ s \in \set{-1, 1}^N}}\!\!\!\!\alpha_{p, s}(0) = \sum_{p \in F} \sigma_{p}(0) = m - \mg(F).\]
\end{enumerate}
\end{lemma}

\begin{proof}
Since $\alpha_{p, s}(r)$ are solutions of the Vertex System, by Lemma~\ref{lemma:coefficients} we can express them as
\[\alpha(r) = Z_{\cc_F(r)^{(0)}}^{-1} \cdot \Big(\sum_{p \in F \setminus \set{q}}\!\!\!\!e^{-d_1(\cc_p(r), q + r t)}\Big)_{q \in F,\,t \in \set{-1, 1}^N},\]
so clearly, they are continuous as functions $\intoo{0}{\frac{\skw(F)}{2}} \to \RR$.

According to Lemma~\ref{lemma:coefficients} they are also determined by the Corner System. Let $B(r)$ denote the $(2^N m \times 2^N m)$-matrix of this system, i.e.\ we have $B(r)\,\alpha(r) = R(r)$ for $R(r) := \Big(\sum_{p \in \cn{q, t}} e^{-d_1(\cc_p(r), q + r u)}\Big)_{q \in F,\,u \in \set{-1, 1}^N}$. Since the solution of this system is unique by Lemma~\ref{lemma:coefficients}, $B(r)$ must be invertible and we have $\alpha(r) = B(r)^{-1} R(r)$.

Clearly, the entries of~$B(r)$ and~$R(r)$ have limits when $r$ goes to~$0$, and we may define $B(0) := \lim_{r \decr 0} B(r)$ and $R(0) := \lim_{r \decr 0} R(r)$. Consider the system $B(0)\,x = R(0)$. Spelling out the individual equations, we get
\[x_{q, u} + \!\!\sum_{\substack{p \in \cn{q, u} \\ s \in \set{-1, 1}^N}}\!\!\!\!x_{p, s}\,e^{-d_1(p, q)} = \sum_{p \in \cn{q, u}} e^{-d_1(p, q)}.\]
Add these equations for a fixed~$q$, letting $u$ range over~$\set{-1, 1}^N$. Taking into account Lemma~\ref{lemma:vertex-directions}~\ref{lemma:vertex-directions:disjoint-union}, we get
\begin{gather*}
\sum_{u \in \set{-1, 1}^N}\!\!\Big(x_{q, u} + \!\!\sum_{\substack{p \in \cn{q, u} \\ s \in \set{-1, 1}^N}}\!\!\!\!x_{p, s}\,e^{-d_1(p, q)}\Big) = \sum_{\substack{p \in F \\ s \in \set{-1, 1}^N}}\!\!\!\!x_{p, s}\,e^{-d_1(p, q)}
\end{gather*}
on the left-hand side, and
\begin{gather*}
\sum_{u \in \set{-1, 1}^N} \sum_{p \in \cn{q, u}} e^{-d_1(p, q)} = \sum_{p \in F \setminus \set{q}} e^{-d_1(p, q)} = \Big(\sum_{p \in F} e^{-d_1(p, q)}\Big) - 1 = \\
=\Big(\sum_{p \in F} e^{-d_1(p, q)}\Big) - \Big(\sum_{p \in F} e^{-d_1(p, q)} w_p\Big) = \sum_{p \in F} e^{-d_1(p, q)} (1 - w_p)
\end{gather*}
on the right-hand side. Denote $y_p := \sum_{s \in \set{-1, 1}^N} x_{p, s}$; we then obtain the system
\[\sum_{p \in F} e^{-d_1(p, q)} y_p = \sum_{p \in F} e^{-d_1(p, q)} (1 - w_p),\]
or in matrix form, $Z_F \cdot y = Z_F \cdot (1_m - w)$. Since $Z_F$ is positive definite and therefore invertible, we may cancel it to obtain $y = 1_m - w$. Plugging this into the original equations for~$x$, we get
\begin{gather*}
x_{q, u} = \Big(\sum_{p \in \cn{q, u}} e^{-d_1(p, q)}\Big) - \Big(\sum_{p \in \cn{q, u}} e^{-d_1(p, q)} y_p\Big) = \\
=\sum_{p \in \cn{q, u}} e^{-d_1(p, q)} (1 - y_p) = \sum_{p \in \cn{q, u}} e^{-d_1(p, q)} w_p.
\end{gather*}
We see that the system $B(0)\,x = R(0)$ has at most one solution. But the calculated candidate is indeed a solution, as demonstrated by the following calculation:
\begin{gather*}
x_{q, u} +\!\!\sum_{\substack{p \in \cn{q, u} \\ s \in \set{-1, 1}^N}}\!\!\!\!x_{p, s}\,e^{-d_1(p, q)} = \Big(\sum_{p \in \cn{q, u}} e^{-d_1(p, q)} w_p\Big) + \!\!\sum_{\substack{p \in \cn{q, u} \\ s \in \set{-1, 1}^N}}\!\!\!\!\Big(\sum_{p' \in \cn{p, s}} e^{-d_1(p', p)} w_{p'}\Big) e^{-d_1(p, q)} = \\
= \Big(\sum_{p \in \cn{q, u}} e^{-d_1(p, q)} w_p\Big) + \sum_{p \in \cn{q, u}}\!\!\Big(e^{-d_1(p, q)}\!\!\sum_{s \in \set{-1, 1}^N} \sum_{p' \in \cn{p, s}} e^{-d_1(p', p)} w_{p'}\Big) = \\
= \Big(\sum_{p \in \cn{q, u}} e^{-d_1(p, q)} w_p\Big) + \sum_{p \in \cn{q, u}}\!\!\Big(e^{-d_1(p, q)} \sum_{p' \in F \setminus \set{p}} e^{-d_1(p', p)} w_{p'}\Big) = \\
= \Big(\sum_{p \in \cn{q, u}} e^{-d_1(p, q)} w_p\Big) + \sum_{p \in \cn{q, u}}\!\!\bigg(e^{-d_1(p, q)} \Big(\big(\sum_{p' \in F} e^{-d_1(p', p)} w_{p'}\big) - w_p\Big)\bigg) = \\
= \Big(\sum_{p \in \cn{q, u}} e^{-d_1(p, q)} w_p\Big) + \sum_{p \in \cn{q, u}}\!\!\Big(e^{-d_1(p, q)} \big(1 - w_p\big)\Big) = \sum_{p \in \cn{q, u}} e^{-d_1(p, q)}.
\end{gather*}
Thus the system $B(0)\,x = R(0)$ has a unique solution, so the matrix $B(0)$ is invertible. We therefore have $\lim_{r \decr 0} \alpha(r) = \lim_{r \decr 0} B(r)^{-1} R(r) = B(0)^{-1} R(0) = x$. The above calculation also shows $\lim_{r \decr 0} \sigma(r) = y = 1_m - w$.
\end{proof}

\begin{remark}
For the strategy of the proof of Lemma~\ref{lemma:coefficient-limit}, it was important to use the Corner System. The same idea would not work with the Vertex System since the limit of its coefficient matrix $Z_{\cc_F(r)^{(0)}}$ is $Z_F$, except with each row and column repeated $2^N$~times, meaning that it is not an invertible matrix.

Before we noticed that we can use the Corner System, we tried to get the limits of $\alpha_{p, s}(r)$ by writing them as quotients of determinants via Cramer's rule and expanding both determinants into power series (doable since each of their entries is $e$ to the power of a linear function in~$r$). While doing this, we noticed a pattern that happened in every numerical experiment we tried, regardless of the number of points or the dimension of the ambient space. We record it here as a conjecture since it could potentially be interesting for other uses.
\end{remark}

\begin{conjecture}
For every skew $F \in \ifsub(\ell_1^N)$ and $r \in \intoo{0}{\frac{\skw(F)}{2}}$ we have
\[\det\big(Z_{\cc_F(r)^{(0)}}\big) = 4^k r^k + O(r^{k+1})\]
where $k := 2^{N-1} N \cdot \card{F}$ (i.e.\ $k$ is the number of $1$-dimensional faces in~$\cc_F(r)$).
\end{conjecture}

We now at last have everything ready to prove the main continuity result of this paper.

\begin{theorem}\label{theorem:skew-magnitude-continuity}
For every $N \in \NN$, the map $\mg[\icsub(\ell_1^N)]$ is continuous at every skew input.

In more detail, for every subset $F \subseteq \ell_1^N$ which is non-empty skew compact (therefore finite, by Proposition~\ref{proposition:skew-implies-discrete}) we have
\[\lim_{r \decr 0} \mg\big(\cc_F(r)\big) = \mg(F),\]
so magnitude is continuous at~$F$ by~Corollary~\ref{corollary:continuity-of-magnitude-with-cubes}.
\end{theorem}

\begin{proof}
Using notation and results from Theorem~\ref{theorem:cubes-weight-measure} and Lemma~\ref{lemma:coefficient-limit}, we obtain
\begin{gather*}
\lim_{r \decr 0} \mg\big(\cc_F(r)\big) = \lim_{r \decr 0} \Big(m\,(1 + r)^N - \!\!\!\!\sum_{\substack{p \in F \\ s \in \set{-1, 1}^N}}\!\!\!\!\alpha_{p, s}(r)\Big) = \\
= \lim_{r \decr 0} \big(m\,(1 + r)^N\big) - \lim_{r \decr 0} \Big(\!\!\!\!\sum_{\substack{p \in F \\ s \in \set{-1, 1}^N}}\!\!\!\!\alpha_{p, s}(r)\Big) = m \ -\!\!\!\!\sum_{\substack{p \in F \\ s \in \set{-1, 1}^N}}\!\!\!\!\alpha_{p, s}(0) = \\
= m - \big(m - \mg(F)\big) = \mg(F).
\end{gather*}
\end{proof}

Let us now illustrate the obtained results on a few simple examples. We will use the formula for the weight measure and magnitude from Theorem~\ref{theorem:cubes-weight-measure}, where we calculate the coefficients $\alpha_{p, s}$ via either the Vertex System or the Corner System, whichever is convenient.

We start with the sanity check that in the case of a single cube, we get back the familiar result.

\begin{example}\label{example:singleton}
Let $F \subseteq \ell_1^N$ be a singleton. Then the right-hand sides of the equations in the Vertex System are sums of zero summands, so equal to~$0$. As the system has a unique solution, all alphas are~$0$. Hence
\[\omega_{\cc_F(r)} \ = \ \tfrac{1}{2^N}\!\!\sum_{D \in \NN_{\leq N}}\!\!\lambda_{C_F(r)^{(D)}}^D,\]
as expected. We then get $\mg\big(\cc_F(r)\big) = (1 + r)^N$ which, when $r$ goes to~$0$, converges to~$1$, i.e.\ magnitude of a singleton.
\end{example}

\begin{example}\label{example:skew-pair}
Let $F \subseteq \ell_1^N$ be a skew set with two elements. Up to isometric isomorphism, we may without loss of generality assume that $F = \set{0, p}$ where $p_k > 0$ for all $k \in \intcc[\NN]{1}{N}$. Then $\skw(F) = \min\set{p_k}{k \in \intcc[\NN]{1}{N}}$.

Let us assume $r \in \intoo{0}{\frac{\skw(F)}{2}}$ so that the two cubes in $\cc_F(r)$ have disjoint projections. Note that $\cn{0, 1_N} = \set{p}$, $\cn{p, -1_N} = \set{0}$, and all other corners are empty. It follows immediately from the Corner System that if $\cn{q, t} = \emptyset$ (which is the case in particular when the vertex $q + r t$ lies on the edge of the smallest rectangle containing $\cc_F(r)$), then $\alpha_{q, t} = 0$. Thus the only potentially non-zero alphas are $\alpha_{0, 1_N}$ and $\alpha_{p, -1_N}$, for which we get the equations
\begin{align*}
\alpha_{0, 1_N} \ + \sum_{s \in \set{-1, 1}^N} \alpha_{p, s}(r)\,e^{-d_1(p + r s, r 1_N)} \ &= \ e^{-d_1(\cc_p(r), r 1_N)}, \\
\alpha_{p, -1_N} \ + \sum_{s \in \set{-1, 1}^N} \alpha_{0, s}(r)\,e^{-d_1(r s, p - r 1_N)} \ &= \ e^{-d_1(\cc_0(r), p - r 1_N)}
\end{align*}
which simplify to
\begin{align*}
\alpha_{0, 1_N} \ + \ \alpha_{p, -1_N}(r)\,e^{-(\|p\|_1 - 2 N r)} \ &= \ e^{-(\|p\|_1 - 2 N r)}, \\
\alpha_{p, -1_N} \ + \ \alpha_{0, 1_N}(r)\,e^{-(\|p\|_1 - 2 N r)} \ &= \ e^{-(\|p\|_1 - 2 N r)}.
\end{align*}
From this we easily get $\alpha_{0, 1_N} = \alpha_{p, -1_N} = \frac{e^{-(\|p\|_1 - 2 N r)}}{1 + e^{-(\|p\|_1 - 2 N r)}}$. Hence
\[\omega_{\cc_F(r)} \ = \ \Big(\tfrac{1}{2^N}\!\!\sum_{D \in \NN_{\leq N}}\!\!\lambda_{\cc_F(r)^{(D)}}^D\Big) - \frac{e^{-(\|p\|_1 - 2 N r)}}{1 + e^{-(\|p\|_1 - 2 N r)}} \big(\delta_{r 1_N} + \delta_{p - r 1_N}\big)\]
and
\[\mg\big(\cc_F(r)\big) = 2 (1 + r)^N - \frac{2 e^{-(\|p\|_1 - 2 N r)}}{1 + e^{-(\|p\|_1 - 2 N r)}}.\]
In the limit we get
\[\lim_{r \decr 0} \mg\big(\cc_F(r)\big) = 2 - \frac{2 e^{-\|p\|_1}}{1 + e^{-\|p\|_1}} = \frac{2}{1 + e^{-\|p\|_1}} = \frac{2}{1 + e^{-d_1(0, p)}}\]
which is the familiar formula for magnitude of two points.

This situation is actually simple enough that we do not need the theory from this paper; we can also calculate $\mg\big(\cc_F(r)\big)$ via two applications of the inclusion-exclusion principle for magnitude~\cite[Proposition 2.3.2]{leinster2010}. Specifically, we have
\begin{gather*}
\mg\big(\cc_0(r) \cup \set{r 1_N, p - r 1_N}\big) = \\
= \mg\big(\cc_0(r)\big) + \mg\big(\set{r 1_N, p - r 1_N}\big) - \mg\big(\set{r 1_N}\big) = \\
= (1 + r)^N + \frac{2}{1 + e^{-(\|p\|_1 - 2 N r)}} - 1,
\end{gather*}
so
\begin{gather*}
\mg\big(\cc_F(r)\big) = \mg\big(\cc_0(r) \cup \set{r 1_N, p - r 1_N} \cup \cc_p(r)\big) = \\
= \mg\big(\cc_0(r) \cup \set{r 1_N, p - r 1_N}\big) + \mg\big(\cc_p(r)\big) - \mg\big(\set{p - r 1_N}\big) = \\
= (1 + r)^N + \frac{2}{1 + e^{-(\|p\|_1 - 2 N r)}} - 1 + (1 + r)^N - 1 = \\
= 2 (1 + r)^N - \frac{2 e^{-(\|p\|_1 - 2 N r)}}{1 + e^{-(\|p\|_1 - 2 N r)}},
\end{gather*}
which is the same result as before.

However, while this trick works for skew two-element sets, it fails in general, as soon as we have at least three points. For example, it would not work in Example~\ref{example:triple} below, where we need to rely on our theory.
\end{example}

\begin{example}\label{example:nonskew-pair}
It is possible to verify magnitude continuity at two-element sets with more standard methods also in the non-skew case. Without loss of generality let $F = \set{0, p} \subseteq \ell_1^N$ where $p = (p_1, \ldots, p_k, 0, \ldots, 0)$ and $p_1, \ldots, p_k > 0$. Let $P \subseteq \ell_1^k$ be the projection of~$F$ onto the first $k$ coordinates. Then $P$ is skew, so by Example~\ref{example:skew-pair}
\begin{gather*}
\mg\big(\cc_F(r)\big) = \mg\big(\cc_P(r) \times_1 \intcc{-r}{r}^{N - k}\big) = \mg\big(\cc_P(r)\big) \cdot \mg\big(\intcc{-r}{r}^{N - k}\big) = \\
= \bigg(2 (1 + r)^k - \frac{2 e^{-(\|p\|_1 - 2 k r)}}{1 + e^{-(\|p\|_1 - 2 k r)}}\bigg) \cdot (1 + r)^{N - k},
\end{gather*}
therefore again
\[\lim_{r \decr 0} \mg\big(\cc_F(r)\big) = \Big(2 - \frac{2 e^{-\|p\|_1}}{1 + e^{-\|p\|_1}}\Big) \cdot 1^{N - k} = \frac{2}{1 + e^{-\|p\|_1}} = \mg(F).\]
\end{example}

This example showed that skewness is not required for magnitude continuity. It is however relevant for the formula for weight measure in Theorem~\ref{theorem:cubes-weight-measure}, as the next example shows.

\begin{example}\label{example:planar-nonskew-pair}
Consider a special case of Example~\ref{example:nonskew-pair} where we take $F = \set{(0, 0), (a, 0)} \subseteq \ell_1^2$ with $a > 0$. The skewness of the subset $\set{0, a} \subseteq \ell_1^1$ is~$a$, so by Example~\ref{example:skew-pair} we get for $r \in \intoo{0}{\frac{a}{2}}$
\begin{gather*}
\omega_{\cc_{\set{0, a}}(r)} = \tfrac{1}{2} \big(\delta_{-r} + \delta_{r} + \delta_{a - r} + \delta_{a + r} + \lambda_{\intcc{-r}{r}}^1 + \lambda_{\intcc{a - r}{a + r}}^1\big) - \frac{e^{-(a - 2r)}}{1 + e^{-(a - 2r)}} \big(\delta_{r} + \delta_{a - r}\big) = \\
= \tfrac{1}{2} \big(\delta_{-r} + \delta_{a + r} + \lambda_{\intcc{-r}{r}}^1 + \lambda_{\intcc{a - r}{a + r}}^1\big) + \frac{1 - e^{-(a - 2r)}}{2 (1 + e^{-(a - 2r)})} \big(\delta_{r} + \delta_{a - r}\big).
\end{gather*}
Hence
\begin{gather*}
\omega_{\cc_F(r)} = \omega_{\cc_{\set{0, a}}(r)} \cdot \tfrac{1}{2} \big(\delta_{-r} + \delta_{r} + \lambda_{\intcc{-r}{r}}^1\big) = \\
= \tfrac{1}{4} \Big(\delta_{(-r, -r)} + \delta_{(-r, r)} + \delta_{(a + r, -r)} + \delta_{(a + r, r)} + \\
+ \lambda_{\set{-r} \times \intcc{-r}{r}}^1 + \lambda_{\intcc{-r}{r} \times \set{-r}}^1 + \lambda_{\intcc{-r}{r} \times \set{r}}^1 + \\
+ \lambda_{\set{a + r} \times \intcc{-r}{r}}^1 + \lambda_{\intcc{a - r}{a + r} \times \set{-r}}^1 + \lambda_{\intcc{a - r}{a + r} \times \set{r}}^1 + \\
+ \lambda_{\intcc{-r}{r} \times \intcc{-r}{r}}^2 + \lambda_{\intcc{a - r}{a + r} \times \intcc{-r}{r}}^2\Big) + \\
+ \frac{1 - e^{-(a - 2r)}}{4 (1 + e^{-(a - 2r)})} \Big(\delta_{(r, -r)} + \delta_{(r, r)} + \delta_{(a - r, -r)} + \delta_{(a - r, r)} + \lambda_{\set{r} \times \intcc{-r}{r}}^1 + \lambda_{\set{a - r} \times \intcc{-r}{r}}^1\Big).
\end{gather*}
Observe that this formula differs from the one in Theorem~\ref{theorem:cubes-weight-measure}, specifically in the coefficients in front of $\lambda_{\set{r} \times \intcc{-r}{r}}^1$ and $\lambda_{\set{a - r} \times \intcc{-r}{r}}^1$.
\end{example}

\begin{example}\label{example:triple}
Let us take a concrete example $F = \set{(0, 0), (4, 8), (7, 3)} \subseteq \ell_1^2$; then $\skw(F) = 3$. Let $r \in \intoo{0}{\frac{3}{2}}$. Recall from Example~\ref{example:systems} that alphas are given by the following Corner System.
{\tiny\setlength{\arraycolsep}{1pt}
\[
\begin{pmatrix}
1 & 0 & 0 & 0 & 0 & 0 & 0 & 0 & 0 & 0 & 0 & 0 \\
0 & 1 & 0 & 0 & 0 & 0 & 0 & 0 & 0 & 0 & 0 & 0 \\
0 & 0 & 1 & 0 & 0 & 0 & 0 & 0 & 0 & 0 & 0 & 0 \\
0 & 0 & 0 & 1 & e^{-(12 - 4r)} & e^{-(12 - 2r)} & e^{-(12 - 2r)} & e^{-12} & e^{-(10 - 4r)} & e^{-(10 - 2r)} & e^{-(10 - 2r)} & e^{-10} \\
e^{-12} & e^{-(12 - 2r)} & e^{-(12 - 2r)} & e^{-(12 - 4r)} & 1 & 0 & 0 & 0 & 0 & 0 & 0 & 0 \\
0 & 0 & 0 & 0 & 0 & 1 & 0 & 0 & 0 & 0 & 0 & 0 \\
0 & 0 & 0 & 0 & 0 & 0 & 1 & 0 & e^{-(8 - 2r)} & e^{-(8 - 4r)} & e^{-8} & e^{-(8 - 2r)} \\
0 & 0 & 0 & 0 & 0 & 0 & 0 & 1 & 0 & 0 & 0 & 0 \\
e^{-10} & e^{-(10 - 2r)} & e^{-(10 - 2r)} & e^{-(10 - 4r)} & 0 & 0 & 0 & 0 & 1 & 0 & 0 & 0 \\
0 & 0 & 0 & 0 & e^{-(8 - 2r)} & e^{-8} & e^{-(8 - 4r)} & e^{-(8 - 2r)} & 0 & 1 & 0 & 0 \\
0 & 0 & 0 & 0 & 0 & 0 & 0 & 0 & 0 & 0 & 1 & 0 \\
0 & 0 & 0 & 0 & 0 & 0 & 0 & 0 & 0 & 0 & 0 & 1
\end{pmatrix}
\cdot
\]
\[
\cdot
\begin{pmatrix}
\alpha_{(0, 0), (-1, -1)}(r) \\
\alpha_{(0, 0), (-1, 1)}(r) \\
\alpha_{(0, 0), (1, -1)}(r) \\
\alpha_{(0, 0), (1, 1)}(r) \\
\alpha_{(4, 8), (-1, -1)}(r) \\
\alpha_{(4, 8), (-1, 1)}(r) \\
\alpha_{(4, 8), (1, -1)}(r) \\
\alpha_{(4, 8), (1, 1)}(r) \\
\alpha_{(7, 3), (-1, -1)}(r) \\
\alpha_{(7, 3), (-1, 1)}(r) \\
\alpha_{(7, 3), (1, -1)}(r) \\
\alpha_{(7, 3), (1, 1)}(r)
\end{pmatrix}
=
\begin{pmatrix}
0 \\
0 \\
0 \\
e^{-(12 - 4r)} + e^{-(10 - 4r)} \\
e^{-(12 - 4r)} \\
0 \\
e^{-(8 - 4r)} \\
0 \\
e^{-(10 - 4r)} \\
e^{-(8 - 4r)} \\
0 \\
0
\end{pmatrix}
\]
}
The exact solution of this system is lengthy, but see the picture below for $r = 1$ where we wrote the numerical value (rounded to seven decimals) for each alpha at its corresponding vertex.
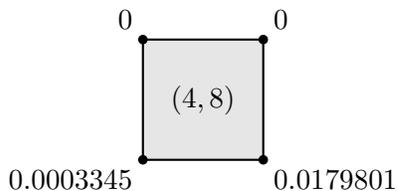
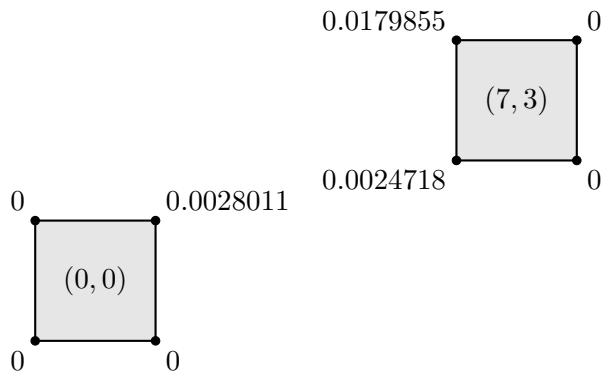
\begin{center}
\begin{tikzpicture}[scale = 0.8]
\filldraw[fill = black!10, thick] (-1, -1) rectangle (1, 1);
\node at (0, 0) {$(0, 0)$};
\filldraw (-1, -1) circle (2pt) node [below left] {$0$};
\filldraw (-1, 1) circle (2pt) node [above left] {$0$};
\filldraw (1, -1) circle (2pt) node [below right] {$0$};
\filldraw (1, 1) circle (2pt) node [above right] {$0.0028011$};
\filldraw[fill = black!10, thick] (3, 7) rectangle (5, 9);
\node at (4, 8) {$(4, 8)$};
\filldraw (3, 7) circle (2pt) node [below left] {$0.0003345$};
\filldraw (3, 9) circle (2pt) node [above left] {$0$};
\filldraw (5, 7) circle (2pt) node [below right] {$0.0179801$};
\filldraw (5, 9) circle (2pt) node [above right] {$0$};
\filldraw[fill = black!10, thick] (6, 2) rectangle (8, 4);
\node at (7, 3) {$(7, 3)$};
\filldraw (6, 2) circle (2pt) node [below left] {$0.0024718$};
\filldraw (6, 4) circle (2pt) node [above left] {$0.0179855$};
\filldraw (8, 2) circle (2pt) node [below right] {$0$};
\filldraw (8, 4) circle (2pt) node [above right] {$0$};
\end{tikzpicture}
\end{center}
As expected, the ``outer'' vertices have their alpha value equal to~$0$. Intuitively, alphas represent the ``reduction in weight'' (compared to if a cube was alone) due to other cubes in the corresponding direction.

Letting $r$ go to~$0$, we get the following approximation for the vector of alphas:
\[\big(0, 0, 0, 0.0000515, 0.0000061, 0, 0.0003353, 0, 0.0000454, 0.0003353, 0, 0\big).\]
Calculating $3$ minus the sum of these values, we get approximately~$2.99923$, which is indeed a numerical approximation for $\mg(F)$.
\end{example}

\section{Conclusion}

Summarizing this paper, we have done the following:
\begin{itemize}
\item
We gave a characterization of continuity at a point for magnitude in tractable spaces (Lemma~\ref{lemma:magnitude-point-continuity-characterization-in-tractable-spaces}).
\item
We gave a formula for weight measure and magnitude of a finite union of cubes in~$\ell_1^N$, assuming that projections of cubes onto coordinate axes are pairwise disjoint (Theorem~\ref{theorem:cubes-weight-measure}).
\item
We proved that magnitude in~$\ell_1^N$ is continuous at least at every skew finite subset (Theorem~\ref{theorem:skew-magnitude-continuity}).
\end{itemize}

It is clear from the definition of skewness that skew finite subsets of~$\ell_1^N$ form an open and dense subset in the set of all finite subsets of~$\ell_1^N$. In this sense, $\mg[\ifsub(\ell_1^N)]$ is continuous ``almost everywhere''.

There is a previously known general magnitude continuity result under the assumption that the domain is star-like~\cite[Theorem 5.4.15.]{LM17}. Our results go in a different direction, with domains being dispersed sets of points.

Skewness was crucially used in our proofs; as we have already seen in Example~\ref{example:planar-nonskew-pair}, the formula for the weight measure in Theorem~\ref{theorem:cubes-weight-measure} does not work in general. Nevertheless (as also suggested by Example~\ref{example:nonskew-pair}) we still believe that magnitude is continuous everywhere in~$\ell_1^N$ (and suspect that it is continuous much more generally --- see Conjectures~5.1 and~5.2 in~\cite{kalisnik2025tractablemetricspacescontinuity}).

To prove magnitude continuity for general finite sets of~$\ell_1^N$, we would want a formula for the union of cubes centered at points for which we allow matching coordinates. Example~\ref{example:planar-nonskew-pair} was easy to calculate, but the general case is much more complicated. We can have chains of points where one point shares some coordinates with the second one, the latter shares some coordinates with the third, but the first and the third do not share coordinates, or share some completely different subset. It is difficult to describe a general case, and this is still part of future work.

We further speculate that if one can derive a useful formula even for unions of cubes which are not necessarily disjoint, this would likely lead to a proof of magnitude continuity at all compact subsets of~$\ell_1^N$ and even Lipschitz continuity when restricted to a bounded part of~$\ell_1^N$ (compare with~\cite[Theorem~4.6]{kalisnik2025tractablemetricspacescontinuity}).

\bibliographystyle{plain}
\bibliography{magnitude}

\end{document}